\documentclass[10pt]{amsart}
\usepackage{amssymb,amsbsy,amsmath,amsfonts,times, amscd}
\usepackage{latexsym,euscript,exscale}
\usepackage{graphicx,color} \usepackage{amsthm}
\usepackage{fancyhdr} \usepackage{url}
\usepackage{epigraph}\usepackage{lmodern}
\usepackage{mathrsfs}\usepackage{cancel}
\usepackage[toc,page]{appendix}\usepackage[T1]{fontenc}
\usepackage{mathtools}\usepackage[all,cmtip]{xy}
\usepackage{enumerate}\usepackage{setspace}
\usepackage{helvet}
\usepackage{float}
\usepackage{hyperref}

\numberwithin{equation}{section}
\newtheorem{theorem}{Theorem}[section]
\newtheorem*{theorem*}{Theorem}
\newtheorem{proposition}[theorem]{Proposition}
\newtheorem{problem}[theorem]{Problem}
\newtheorem*{proposition*}{Proposition}
\newtheorem{lemma}[theorem]{Lemma}
\newtheorem*{lemma*}{Lemma}
\newtheorem{corollary}[theorem]{Corollary}
\newtheorem*{corollary*}{Corollar}

\newtheorem*{fact*}{Fact}
\theoremstyle{definition}
\newtheorem{definition}[theorem]{Definition}
\newtheorem*{definition*}{Definition}
\newtheorem{claim}[theorem]{Claim}
\newtheorem*{claim*}{Claim}

\newtheorem*{conjecture*}{Conjecture}

\newtheorem{theoremi}{Theorem}
\newtheorem{corollaryi}[theoremi]{Corollary}

\newtheorem{example}[theorem]{Example}
\newtheorem*{example*}{Example}
\newtheorem{remark}[theorem]{Remark}

\newtheorem*{remark*}{Remark}

\newtheorem*{note*}{Note}
\newtheorem*{question*}{Question}

\newcommand {\M}{\mathbb M}

\newcommand{\eps}{\varepsilon}

\newcommand{\cD}{\mathcal{D}}

\newcommand{\cU}{\mathcal{U}}

\newcommand{\N}{\mathbb{N}}
\newcommand{\R}{\mathbb{R}}

\newcommand{\Lip}{\mathrm{Lip}}
\newcommand{\co}{\text{cof}}

\newcommand{\Exp}{\mathrm{Exp}}

\begin{document}

 \title{Asymptotic coarse Lipschitz equivalence}
 
\author[B. M. Braga]{Bruno M. Braga}
\address[B. M. Braga]{IMPA, Estrada Dona Castorina 110, 22460-320, Rio de Janeiro, Brazil}
\email{demendoncabraga@gmail.com}
\urladdr{https://sites.google.com/site/demendoncabraga/}
\thanks{B. M. Braga was partially supported by NSF Grant   DMS 2054860.}

\author[G. Lancien]{Gilles Lancien}
\address[G. Lancien]{ Universit\'e de Franche-Comt\'e, Laboratoire de math\'ematiques de Besan\c con, CNRS UMR-6623, 16 route de Gray, 25000 Besan\c con, France}
\email{gilles.lancien@univ-fcomte.fr}

\maketitle 

\begin{abstract} We introduce the notion of asymptotic coarse Lipschitz equivalence of metric spaces. We show that  it is strictly weaker than  coarse Lipschitz equivalence. We study its impact on the asymptotic dimension of metric spaces. Then we focus on Banach spaces. We prove that, for $2\le p<\infty$, being linearly isomorphic to $\ell_p$ is stable under asymptotic coarse Lipschitz equivalences. Finally, we establish a version of the Gorelik principle in this setting and apply it to prove the stability of various properties of asymptotic uniform smoothness of Banach spaces under asymptotic coarse Lipschitz equivalences.
\end{abstract}

\section{Introduction}\label{SectionIntro}

As Banach spaces are normed linear spaces, linear isometries    are the correct kind of equivalence if one wants to keep a complete track of all aspects of Banach spaces. After linear isometries,  linear isomorphisms are the next natural equivalence between Banach spaces which still keeps track of their linear structure; and this is  usually the notion of equivalence Banach space theorists are more concerned with. However, Banach spaces are in particular metric spaces, so (nonlinear) isometries and even  Lipschitz equivalences are also natural notions to consider. For a long time, understanding the minimal properties a certain notion of equivalence can have and still generate an interesting theory  has been the topic of intensive  study for researches in Banach space theory (e.g., \cite{MazurUlam1932,Enflo1970Israel,Ribe1976ArkMat,Ribe1984Israel,JohnsonLindenstraussSchechtman1996GAFA,Kalton2012MathAnn}). The first milestone in this area is probably the Mazur-Ulam's theorem, which says that (nonlinear) isometries between (real) Banach spaces which preserve zero are automatically linear (see \cite{MazurUlam1932}). In the 70's, P. Enflo showed that the linear structure of Hilbert spaces is determined by a much weaker notion: as long as a Banach space $X$ is uniformly equivalent to $\ell_2$,\footnote{I.e., there is a bijection $X\to \ell_2$ which is uniformly continuous and so is its inverse.} it must be (linearly) isomorphic to it (see \cite[Theorem 6.3.1]{Enflo1970Israel}). In the last few decades, researches have been interested in an even weaker notion, called \emph{coarse Lipschitz equivalence}, which is an equivalence that only takes into account large scale geometric aspects of the spaces (the precise definition will be given in the next paragraph). As we explain in details below, this paper concerns an even weaker notion of large scale equivalence between metric spaces which turns out to still be strong enough so that many linear aspects of Banach spaces can be recovered by it.

The large scale geometry of metric spaces is often regarded as the study of metric spaces by  observers positioned very far away from the objects of interest. In this sense, the local geometric aspects are not of interest and global aspects are the focus of the study. One of the main consequences of this approach is that different events happening at a uniformly bounded distance from each other should be treated as being ``morally'' the same. To formalize this, we have the concept of \emph{closeness} between maps into a metric space. Precisely, given a set $X$, a metric space $(Y,\partial)$, and maps $f,g:X\to Y$, we say that $f$ is \emph{close} to $g$, and write $f\sim g$, if 
\[\sup_{x\in X}\partial(f(x),g(x))<\infty.\]
So, $\sim$ is an equivalence relation between maps $X\to Y$ which weakens the relation of equality.  Moreover, if $(X,d)$ is also a metric space, then a function $f:X\to Y$ has a modulus of uniform continuity $\omega_f:[0,\infty)\to [0,\infty)$ given by 
\[\omega_f(t)=\sup\{\partial(f(x),f(x'))\mid d(x,x')\leq t\}\ \text{ for all } \ t\ge 0.\]
  The map $f$ is then called \emph{coarse Lipschitz} if $\omega_f$ is bounded above by an affine function, i.e., if there is $L>0$ such that \[\omega_f(t)\leq Lt+L \ \text{ for all }\ t\geq 0.\] The map $f$ is called a \emph{coarse Lipschitz equivalence} if it is coarse Lipschitz and there is a coarse Lipschitz map $g:Y\to X$ such that $g\circ f$ and $f\circ g$ are close to the identities $\mathrm{Id}_X$ and $\mathrm{Id}_Y$, respectively. A coarse Lipschitz equivalence is also often referred to as a \emph{quasi-isometry} in the literature; especially in geometric group theory.

The main focus of this paper is to study a weakening of the closeness relation as well as the large scale geometry induced by it. Moreover, our study will have a special focus on Banach spaces. Precisely, the following is the main novel definition in these notes.

\begin{definition}
Let $(X,d)$ and $(Y,\partial)$ be metric spaces. We say that maps $f,g:X\to Y$ are \emph{asymptotically close}, and write $f\sim_\infty g$ if, for some $x_0\in X$, we have that
\[\lim_{x\to \infty}\frac{\partial (f(x),g(x))}{d(x,x_0)}=0.\]
We adopt the convention that the above limit is always $0$, when $X$ is a bounded metric space. In other words, if $X$ is bounded, two maps $f,g:X \to Y$ are always asymptotically close. 
\end{definition}

It is evident  that the definition above is weaker than the one of closeness and that it is independent of the point $x_0$ above (Proposition \ref{PropDefiClose}).

\begin{definition}\label{DefiAsympCoaLipEquiv}
We say that metric spaces $X$ and $Y$ are \emph{asymptotically coarse Lipschitz equivalent} to each other if there are coarse Lipschitz maps $f:X\to Y$ and $g:Y\to X$ such that $g\circ f$ and $f\circ g$ are asymptotically close to $\mathrm{Id}_X$ and $\mathrm{Id}_Y$, respectively. The map $f$ (and $g$) is called an \emph{asymptotic coarse Lipschitz equivalence}.

Note that if $X$ and $Y$ are bounded, the above assumptions are satisfied by any maps $f:X\to Y$ and $g:Y\to X$.
\end{definition}

We show that asymptotic coarse Lipschitz equivalence is an equivalence relation in the class of metric spaces (Proposition \ref{PropTransitivity}). Notice that a  coarse Lipschitz equivalence $f$ is always an asymptotic coarse Lipschitz equivalence. The converse implication is however not true, as we see in   Example  \ref{ExWeakCoarseEqNotCoarseEqBanSp} below. In fact,   being asymptotic coarse Lipschitz equivalent is \emph{strictly} weaker than being  coarse Lipschitz equivalent: there are   asymptotically coarse Lipschitz equivalent  metric spaces $X$ and $Y$ which are not even   coarsely equivalent;  see Example \ref{ExampleAsyCoarseLipEqNotPreserveAsyDim} for an example where $\mathrm{AsyDim}(X)\neq \mathrm{AsyDim}(Y)$ and  Proposition \ref{PropAsympEquiImpliesCoarseEmb} for a more sophisticated example where $\mathrm{AsyDim}(X)=\mathrm{AsyDim}(Y)$ (where $\mathrm{AsyDim}$ stands for asymptotic dimension, see Subsection \ref{SubsectionSomeEx} for definitions).

\medskip Then, we focus on asymptotic coarse Lipschitz equivalences between Banach spaces. First, in Sections \ref{ultraproducts} and \ref{SectionAsyCLStruclp} we prove. 
 
\begin{theoremi} [Proved as Corollary \ref{CorAsyCLStructurelp} below]
If a Banach space $X$ is asymptotically coarse Lipschitz equivalent to $\ell_p$, for $p\in [2,\infty)$, then $ X$ is linearly isomorphic to $\ell_p$.  \label{ThmAsyCLStructurelpINTRO}
\end{theoremi}

In fact, we prove a stronger result. We show that the asymptotically coarse Lipschitz strucutre of  $\ell_{p_1}\oplus \ldots \oplus \ell_{p_n}$ completely determines its linear structure for all $p_1,
\ldots, p_n\in [2,\infty)$ (see Theorem \ref{ThmAsyCLStructurelp}). Our methods are not enough to obtain the analogous result for $p$'s smaller than 2. This happens since the usual midpoint argument used in this range is not enough for this context (see Remark \ref{RemarkMidPoint} below for a detailed discussion about that).

 Finally, in Section \ref{asymptoticstructure}, we establish a variant of the Gorelik principle for asymptotic coarse Lipschitz equivalences. We apply it to show that some asymptotic linear properties of Banach spaces are also preserved under asymptotic coarse Lipschitz equivalences.\footnote{We point out to the reader that the word ``asymptotic'' may be a little misleading here. Indeed, while our choice for this word in the definition of our equivalences is motivated by the notion of asymptotic closeness, $\sim_\infty$, \emph{asympotitic properties} of Banach spaces are usually properties which depend on finite codimensional subspaces only. } To state our main result along these lines, we start recalling the definition of asymptotic uniform smoothness. Given a Banach space $X$, we denote by $B_X$ its closed unit ball, $S_X$ its unit sphere and $\co(X)$ the set of its closed finite codimensional subspaces. Then,  we define the \emph{modulus of asymptotic uniform smoothness of $X$} by letting
\[\bar\rho_X(\tau)=\sup_{x\in S_X}\inf_{Y\in \co(X)}\sup_{y\in S_Y}\|x+\tau y\|-1\]
for all $\tau\geq 0$. The Banach space $X$ is called  \emph{asymptotically uniformly smooth}, abbreviated as \emph{AUS}, if 
\[\lim_{\tau\to 0}\frac{\bar\rho_X(\tau)}{\tau}=0.\]
It is known that if $X$ is AUS, then there is $p\in (1,\infty)$ such that $\bar\rho_X(\tau)\leq C\tau^p$ for all $\tau\geq(0,1)$ (see \cite{KnaustOdellSchlumprecht1999Positivity} or \cite[Theorem 2.1]{Raja2013JFA}); in this case, $X$ is called \emph{$p$-asymptotically uniformly smooth}, abbreviated as \emph{$p$-AUS}. Let us also say that $X$ is $\infty$-AUS (or \emph{asymptotically uniformly flat}, AUF in short) if there exists $\tau_0>0$ such that $\bar\rho_X(\tau)=0$ for all $\tau<\tau_0$. 

Notice that AUSness and $p$-AUSness are  not   isomorphic properties, so equivalent norms may disagree on this matter. If a Banach space $X$ has an equivalent norm making it AUS, we say that $X$ is \emph{AUSable}. We define $p$-AUSable analogously.

The following is our main result about preservation of asymptotic structures. 
 
\begin{theoremi}[]
Let $p\in (1,\infty]$ and let $X$ be a $p$-AUS  Banach space. If a Banach space $Y$    is  asymptotically coarse Lipschitz equivalent to $X$, then $Y$ is $p'$-AUSable for all $p'\in (1,p)$.\label{ThmAUSPreservationINTRO}
\end{theoremi}

We point out that Theorem \ref{ThmAUSPreservationINTRO} cannot be improved to say that $p$-AUSness is preserved by asymptotically coarse Lipschitz equivalence. Indeed, N. Kalton showed that $p$-AUSness is not stable even under the stronger notion of coarse Lipschitz equivalence (see \cite[Theorem 5.4 and Remark in page 170]{Kalton2013Examples}).

\begin{corollaryi}\label{CorAUSPreservationINTRO}
AUSableness is preserved by asymptotically coarse Lipschitz equivalences.
\end{corollaryi}

We finish this introduction with a quick discussion about Theorem \ref{ThmAUSPreservationINTRO} above. For $p\in (1,\infty]$, let us denote  by $\textsf{T}_p$ the class of all $p$-AUSable Banach spaces. In the literature of Banach spaces, one  often studies some small variations of the class $\textsf{T}_p$; this is important both to  better understand $\textsf{T}_p$ as well as to  pin down precise asymptotic properties which are preserved by different notions of equivalences between Banach spaces. For $p\in (1,\infty]$, some of those variations on the definition of $p$-AUSness gives rise to   classes of Banach spaces denoted by  $\textsf{A}_p$ and $\textsf{N}_p$  (we refer the reader to Subsection \ref{SubsectionAsyProp} for precise definitions). Those classes are related by the following inclusions 
\[\textsf{T}_p \subsetneq \textsf{A}_p \subsetneq \textsf{N}_p \subsetneq \bigcap_{p'<p}\textsf{T}_{p'},\ \text{if}\ p\in (1,\infty)\ \ \text{and}\ \ \textsf{T}_\infty \subsetneq \textsf{A}_\infty = \textsf{N}_\infty \subsetneq \bigcap_{p'<\infty}\textsf{T}_{p'}.\]
In  Section \ref{asymptoticstructure}, we actually prove a stronger technical result which, together with the inclusions above, implies Theorem \ref{ThmAUSPreservationINTRO}. Precisely, we shall   prove  that, for $p\in (1,\infty]$, both classes $\textsf{A}_p$ and $\textsf{N}_p$ are preserved by asymptotically coarse Lipschitz equivalences (see Theorem \ref{precise stability}).

\section{Preliminaries}

\subsection{Basic properties} In this subsection, we prove several basic properties of the new definitions given in the introduction. The material proved in here will be used throughout the rest of the paper and it serves as a warm up for the reader to get used to those new definitions. We start with a terminology which will be useful.

\begin{definition}
Let $X$ and $Y$ be asymptotically coarse Lipschitz equivalent metric spaces and let $f:X\to Y$ and $g:Y\to X$ witness this equivalence, i.e., $f$ and $g$ are as in Definition \ref{DefiAsympCoaLipEquiv}. We say that $g$ is an \emph{asymptotic coarse Lipschitz inverse of $f$} and vice-versa.
\end{definition}

\begin{proposition}\label{PropDefiClose}
Let $(X,d)$ and $(Y,\partial)$ be metric spaces, and let  $f,g:X\to Y$ be  maps. The following are equivalent:

\begin{enumerate}
\item\label{PropDefiClose.Item1} For some $x_0\in X$, we have that
\[\lim_{x\to \infty}\frac{\partial (f(x),g(x))}{d(x,x_0)}=0.\]
\item\label{PropDefiClose.Item2} For all $x_0\in X$, we have that
\[\lim_{x\to \infty}\frac{\partial (f(x),g(x))}{d(x,x_0)}=0.\]
\end{enumerate} 
\end{proposition}

\begin{proof}
Suppose \eqref{PropDefiClose.Item1} holds for $x_0\in X$ and fix $x_1\in X$. First, if $X$ has finite diameter, the result follows immediately from our convention. If not, then, for any $x\in X$ with $d(x,x_0)> 2d(x_1,x_0)$, we have that 
\[\frac{\partial(f(x),g(x))}{d(x,x_1)}\leq \frac{\partial(f(x),g(x))}{d(x,x_0)-d(x_1,x_0)}\leq 2  \frac{\partial(f(x),g(x))}{d(x,x_0)}\]
 and the result follows.
\end{proof}

As mentioned in the introduction, asymptotic coarse Lipschitz equivalence is an equivalence relation on the class of metric spaces. Since reflexivity and symmetry are evident, we only need to notice its transitivity. For that, we start by proving a preliminary result which will be useful throughout these notes.

\begin{proposition}\label{PropAlmCLInv}
Let $(X,d)$ and $(Y,\partial)$ be metric spaces,  and let $f:X\to Y$ and $g:Y\to X$ be coarse Lipschitz maps with $g\circ f\sim_\infty\mathrm{Id}_X$.  For each $x_0\in X$ and $\theta>0$,  there is $L>0$, such that for  all $x,x'\in X$ we have that \[d(x,x')>  \theta\max\{  d(x,x_0),  d(x',x_0)\}\ \text{ implies }\ \partial(f(x),f(x'))\geq \frac{1}{L}d(x,x')-L.\]
\end{proposition}

\begin{proof}
Suppose the proposition fails for $x_0\in X$ and $\theta>0$. So, there are sequences $(x_n)_n,(x_n')_n\subset X$  such that
\begin{enumerate}

\item\label{Enume3}  $d(x_n,x'_n)>   \theta\max\{   d(x_n,x_0),   d(x'_n,x_0)\}$ for all $n\in\N$, and 
\item \label{Enume4} $\partial(f(x_n),f(x'_n))\leq \frac{1}{n}d(x_n,x'_n)-n$ for all $n\in\N$. 
\end{enumerate}   
In particular, \eqref{Enume4} implies that $\lim_{n}d(x_n,x_n')=\infty$.

Let $g:Y\to X$ be  coarse Lipschitz and such that $g\circ f\sim_\infty\mathrm{Id}_X$ and fix $M>0$ such that 
\[d(g(y),g(y'))\leq M\partial (y,y')+M\]
for all $y,y'\in Y$. So,  \eqref{Enume4} above gives that
\[d(g(f(x_n)),g(f(x'_n)))\leq \frac{M}{n}d(x_n,x'_n)+M \]
for all $n\in\N$.\\

\begin{claim}
Passing to a subsequence if necessary, we can  assume that \[d(g(f(x_n)),x_n ) \leq \frac{d(x_n,x_n')}{3 }\ \text{ and }\ d(g(f(x_n')),x_n ') \leq \frac{d(x_n,x_n')}{3 }\]
for all $n\in\N$.
\end{claim}

 \begin{proof}
 It is enough to show the claim holds for $(x_n)_n$.  Suppose first that $(x_n)_n$ is bounded. In this case, as $f$ and $g$ are coarse, we must have  that $(d(g(f(x_n)),x_n))_n$ is bounded. Then, as  $\lim_{n}d(x_n,x_n')=\infty$, there is $n_0\in\N$ such that 
\[d(g(f(x_n)),x_n ) \leq \frac{d(x_n,x_n')}{3 }\]
for all $n>n_0$. 

Suppose now $(x_n)_n$ is unbounded. Then, passing to a subsequence, we assume that $\lim_nd(x_n,x_0)=\infty$.   As $g\circ f\sim_\infty\mathrm{Id}_X$,  there is $n_0\in\N$ such that  
\[ d(g(f(x_n)),x_n) \leq \theta\frac{d(x_n,x_0)}{3 }\]
for all $n>n_0$. As $d(x_n,x_n')$ is larger than  $\theta {d(x_n,x_0)}$,  we conclude that 
\[ d(g(f(x_n)),x_n) \leq  \frac{d(x_n,x_n')}{3 }\]
 for all $n>n_0$. 
 \end{proof}

Passing to a subsequence, we now assume the previous claim holds for $(x_n)_n$ and $(x'_n)_n$.  We then conclude that 
 \begin{align*}
\frac{M}{n}d(x_n,x'_n)+M&\geq d(g(f(x_n)),g(f(x'_n)))\\ &\geq d(x_n,x'_n)-d(g(f(x_n)),x_n)-d(g(f(x_n')),x'_n)\\
&\geq \frac{d(x_n,x_n')}{3}
\end{align*}
for all $n>n_0$. This gives us a   contradiction since $\lim_{n}d(x_n,x_n')=\infty$. 
\end{proof}

Next we isolate an immediate corollary of Proposition \ref{PropAlmCLInv}.
\begin{corollary}\label{CorPointwiseExpanssion}
Let $(X,d)$ and $(Y,\partial)$ be metric spaces and $f:X\to Y$ be an asymptotic coarse Lipschitz equivalence. Then, for all $x_0\in X$ there is $L>0$ such that 
\[\partial(f(x),f(x_0))\geq \frac{1}{L}d(x,x_0)-L\]
for all $x\in X$. In particular, for all $x_0\in X$ we have that \[\lim_{x\to \infty} \partial(f(x),f(x_0))=\infty.\]
\end{corollary}

\begin{proposition}\label{PropTransitivity}
The asymptotically coarse Lipschitz equivalence is an equivalence relation in the class of metric spaces. 
\end{proposition}

\begin{proof}
Reflexivity and symmetry of this relation  are evident; so, we only prove its transitivity. For this, let $(X,d_X)$, $(Y,d_Y)$, and $(Z,d_Z)$ be metric spaces, and let $f:X\to Y$ and $g: Y\to Z$   be asymptotic coarse Lipschitz equivalences with asymptotic coarse Lipschitz inverses $f':Y\to X$ and $g':Z\to Y$, respectively. As $f$ and $f'$ are coarse Lipschitz,   fix $L>0$ such that for all $t\ge 0$, 
\[\omega_f(t)\leq Lt+L \text{ and }\ \omega_{f'}(t)\leq Lt+L.\] 
Let us show that $f'\circ g'\circ g\circ f\sim_\infty \mathrm{Id}_X$. 
For that, fix $x_0\in X$ and notice that 
 \begin{align*}
 \frac{d_X(f'(g'(g(f(x)))),x)}{d_X(x,x_0)}&\leq  \frac{d_X(f'(g'(g(f(x)))),f'(f(x)))+ d_X(f'(f(x)),x)}{d_X(x,x_0)} \\
 &\leq \frac{Ld_Y(g'(g(f(x))),f(x))+L+ d_X(f'(f(x)),x)}{d_X(x,x_0)}
\end{align*}  
for all $x\in X$.  By Corollary \ref{CorPointwiseExpanssion}, we have that $\lim_{x\to \infty}d_Y(f(x),f(x_0))=\infty$. Hence, since $g'\circ g\sim_\infty \mathrm{Id}_{Y}$, it follows that
\begin{align*}\limsup_{x\to \infty} \frac{d_Y(g'(g(f(x))),f(x))}{d_X(x,x_0)}&=\limsup_{x\to \infty}\left( \frac{d_Y(g'(g(f(x))),f(x))}{d_Y(f(x),f(x_0))}\cdot \frac{d_Y(f(x),f(x_0))}{d_X(x,x_0)}\right)\\
&\le \limsup_{x\to \infty}\left( \frac{d_Y(g'(g(f(x))),f(x))}{d_Y(f(x),f(x_0))}\cdot \frac{Ld_X(x,x_0)+L}{d_X(x,x_0)}\right)\\
&=0.\end{align*}
Therefore, as $f'\circ f\sim_\infty \mathrm{Id}_X$, we conclude that 
\[\lim_{x\to x_0} \frac{d_X(f'(g'(g(f(x)))),x)}{d_X(x,x_0)}=0.\]
This shows that 
$f'\circ g'\circ g\circ g\sim_\infty \mathrm{Id}_X$ and a completely symmetric argument shows that $g\circ f\circ f'\circ g'\sim_\infty \mathrm{Id}_Y$. Since $g \circ f$ and $f' \circ g'$ are coarse Lipschitz, we conclude that $X$ and $Y$ are coarse Lipschitz equivalent.
\end{proof}

 \subsection{Some examples}\label{SubsectionSomeEx}
Obviously, any coarse Lipschitz equivalence is an asymptotic coarse Lispchitz equivalence. In this subsection, we provide some nontrivial examples of asymptotic coarse Lipschitz equivalences, i.e., examples which are not coarse Lipschitz equivalences (see also Subsection \ref{SubsectionRobEx} for  other such example).  In fact, the examples will be of spaces which are not even coarsely equivalent: recall, $(X,d)$ and $(Y,\partial)$ are \emph{coarsely equivalent} if there are maps $f:X\to Y$ and $g:Y\to X$ such that both $f$ and $g$ are \emph{coarse}, i.e., $\omega_f(t),\omega_g(t)<\infty$ for all $t\geq 0$, and $g\circ f\sim\mathrm{Id}_X$ and $f\circ g\sim \mathrm{Id}_Y$.

\begin{example}
\label{ExWeakCoarseEqNotCoarseEqBanSp}
Asymptotic coarse Lipschitz equivalences do not need to be coarse (Lipschitz) embeddings even in the class of Banach spaces; recall, a \emph{coarse embedding} (resp. \emph{coarse Lipschitz embedding}) is a coarse equivalence (resp.  coarse Lipschitz equivalence with a subset.  Let $f:\R\to \R$ be the identity map and $g:\R\to \R$ be the continuous piecewise affine  function determined by the following properties: 
\begin{enumerate}
\item $g(x)=x$ for all $x\in [-2,2]$,
\item $g$ is constant on the intervals $[2^n,2^n+n]$ and $[-2^n-n,-2^n]$ for all $n\in\N$, and
\item $g'(x)=1$ on the intervals $(2^n+n, 2^{n+1})$ and $( -2^{n+1},-2^n-n)$ for all $n\in\N$. 
\end{enumerate}
Both $f$ and $g$ are clearly  $1$-Lipschitz and \[|x-f(g(x))|=|x-g(f(x))|
\leq 1+\ldots + n\]
for all $x$   in either  $[2^n,2^{n+1}]$ or $[-2^{n+1},-2^{n}]$. Therefore, \[ f\circ g\sim_\infty \mathrm{Id}_X\ \text{ and }\ g\circ f\sim_\infty \mathrm{Id}_Y.\]
However, it is clear that 
\[\inf_{\|x-x'\|=t}  \|g(x)-g(x ')\|=0\]
for all $t>0$. So, $g$ cannot be a coarse Lipschitz embedding even though $g$ is an asymptotic coarse Lipschitz equivalence. Moreover, it is clear that $f$ can be modified (similarly to the definition of $g$), so that $f$ is not a coarse Lipschitz embedding either. 
\end{example}

\begin{remark}\label{RemarkAlmUnco}
A map $f:X\to Y$ between Banach spaces  is   called \emph{almost uncollapsed} if there is $t>0$ such that 
\[\inf_{\|x-x'\|=t}  \|f(x)-f(x ')\|>0;\]
see \cite{Rosendal2017Sigma} and \cite{Braga2018JFA} for more on almost uncollapsed maps (we restrict this definition to Banach spaces so that the condition ``$d(x,x')=\|x-x'\|=t$'' is not  vacuous). Notice that the requirement on the map $f$ not collapsing distances in the sense above is much weaker than the one of $f$ being a coarse (Lipschitz) embedding (see \cite[Proposition 2.5]{Braga2018JFA}).  Example \ref{ExWeakCoarseEqNotCoarseEqBanSp} shows much more than the fact that asymptotic coarse Lipschitz equivalences do not need to be coarse embeddings: they do not even need to be almost uncollapsed. 
\end{remark}

Although Example \ref{ExWeakCoarseEqNotCoarseEqBanSp} shows that an asymptotic coarse Lipschitz equivalence between $X$ and $Y$  does not need to be a coarse Lipschitz equivalence, the spaces in this example, i.e., $X=Y=\R$, are obviously coarse Lipschitz equivalent. In the next example, we show that this does not need to be the case: asymptotic coarse Lipschitz equivalence of metric spaces is strictly weaker than coarse Lipschitz equivalence. Before presenting the example, we recall the definition of asymptotic dimension.  Recall that, if $(X,d)$ is a metric space and $n\in\N\cup\{0\}$, then \emph{$X$ has asymptotic dimension at most $n\in\N$} if for all $r>0$ there are families $\cU_0,\ldots, \cU_n$ of subsets of $X$ such that 
 \begin{enumerate}
 \item $X=\bigcup_{i=0}^n\bigcup_{U\in \cU_i}U$,
 \item $d(U,V):=\inf\{d(x,y),\ (x,y)\in U\times V\}>r$ for all $i\in \{0,\ldots, n\}$ and all distinct  $U,V\in \cU_i$, and 
 \item $\sup_{U\in \cU_i}\mathrm{diam}(U)<\infty$ for all $i\in \{0,\ldots, n\}$.
 \end{enumerate}
The \emph{asymptotic dimension of $X$}, denoted by $\mathrm{AsyDim}(X)$, is then defined to be the minimal $n\in\N\cup\{0\}$ such that $X$ has asymptotic dimension at most $n$ (we  refer the reader to \cite{NowakYuBook} for more on asymptotic dimension).

\begin{example}\label{ExampleAsyCoarseLipEqNotPreserveAsyDim}
Consider   \[X=\{2^n\mid n\in\N\}\ \text{  and }\ Y=\{(2^n,j)\mid n\in\N,\ j\in\{0,1,\ldots, n\}\}\] endowed with their canonical metrics. Then $X$ and $Y$ are asymptotically coarse Lipschitz equivalent; but not coarse Lipschitz equivalent, in fact, they are not even coarsely equivalent.  Indeed, the maps $f:X\to Y$ and $g:Y\to X$ defined by setting \[f(2^n)=(2^n,0)\ \text{ and }\ g(2^n,j)=2^n\]
for all $n\in\N$ and all $j\in \{0,\ldots, n\}$ are clear witnesses of the  asymptotic coarse Lipschitz equivalence of $X$ and $Y$.  In particular, this implies that the asymptotic dimension of metric spaces is not preserved under asymptotic coarse Lipschitz equivalence. Indeed $X$ has asymptotic dimension $0$  and  $Y$ has asymptotic dimension  $1$, while asymptotic dimension is preserved by coarse Lipschitz embeddings (\cite[Proposition 2.2.4 and Theorem 2.2.5]{NowakYuBook}). 
\end{example}

\begin{remark} As, under the optics of coarse geometry,   $\{2^n\mid n\in\N\}$  is the smallest unbounded metric space\footnote{The word ``smallest'' is appropriate here since any unbounded metric space contains a coarse copy of $\{2^n\mid n\in\N\}$.}, Example \ref{ExampleAsyCoarseLipEqNotPreserveAsyDim} shows that containing isometric copies of $(\{0,\ldots, n\})_{n=1}^\infty$ imposes no restriction for asymptotic coarse Lipschitz equivalences. On the other hand, containing a coarse copy of $\N$ does impose some restriction. Indeed, it follows immediately from Corollary \ref{CorPointwiseExpanssion} that if $\N$ coarsely embeds into $X$ and $X$ is asymptotically coarse Lipschitz equivalent to $Y$, then there is a coarse (and therefore Lipschitz) map from $\N$ to $Y$ with unbounded image. It is easy to see that such a map from $\N$ to  $\{2^n\mid n\in\N\}$ does not exist. Therefore, $\{2^n\mid n\in\N\}$ cannot be asymptotically coarse Lipschitz equivalent to any metric space  containing   a coarse copy of $\N$. 
\end{remark}

\subsection{A more robust example}\label{SubsectionRobEx}

As we have seen in Example \ref{ExampleAsyCoarseLipEqNotPreserveAsyDim}, the notion of asymptotic coarse Lipschitz equivalence is strictly weaker than the one of coarse Lipschitz equivalence. In this subsection, we take a deeper look at this and accentuate this difference even further. In fact, Example \ref{ExampleAsyCoarseLipEqNotPreserveAsyDim} 
shows that there are metric spaces $X$ and  $Y$ which are not coarse Lipschitz equivalent, but are asymptotically coarse Lipschitz equivalent. However, the space $Y$ presented therein does not even coarsely embed into $X$. Moreover, it is easy to notice that, for such spaces, there is not even a coarse map $Y\to X$ which is also  uncollapsed.\footnote{A map $f:(X,d)\to (Y,\partial)$ is \emph{ uncollapsed} if there is $r>0 $ such that $\inf_{d(x,x')\geq r}\partial(f(x),f(x'))>0$ (cf. Remark \ref{RemarkAlmUnco}).}  In this subsection, we show that we can take asymptotic coarse Lipschitz equivalences to be much more rigid and still not imply that the metric spaces are even coarsely equivalent to each other (see Proposition \ref{PropWeakerEqForMetric}).

 Let $(X,d)$ and $(Y,\partial)$ be metric spaces, and let  $f:X\to Y$ be a map. For each $s>0$,  
\[\Exp_s(f)=\inf\left\{\frac{d(f(x),f(z))}{d(x,z)}\mid d(x,z)\geq s\right\}\]
(here we use the convention that $\inf\emptyset=\infty$), and we let
\[ \Exp_\infty(f)=\sup_{s>0}\Exp_s(f).\]
Notice that, since   $\Exp_s$ increases as $s$ increases, we have   $\Exp_\infty(f)=\lim_{s\to \infty}\Exp_s(f)$.

\begin{proposition}\label{PropAsympEquiImpliesCoarseEmb}
Let $X$ and $Y$ be metric spaces, and $f:X\to Y$ and $g:Y\to X$ be coarse Lipschitz maps. If   $\Exp_\infty(g\circ f)>0$,  then $f$ is a coarse Lipschitz embedding.\footnote{A completely analogous proof will show that if $f$ and $g$ are only assumed to be coarse, then $f$ is a coarse embedding.}
\end{proposition}

\begin{proof}
 If $f$ is not a coarse Lipschitz embedding, there are $(x_n)_n$ and $(z_n)_n$  in $X$ such that  
\[ \partial(f(x_n ),f(z_n))\leq \frac{1}{n} d(x_n,z_n)-n\]
for all $n\in\N$. In particular, $\lim_nd(x_n,z_n)=\infty$. As $g$ is coarse Lipschitz, there is $L>0 $ so that $\omega_g(t)\leq Lt+L$ for all $t>0$. Hence, 
\begin{align*}
\frac{d(g(f(x_n)),g(f(z_n)))}{d(x_n,z_n)}\leq  \frac{ L n^{-1} d(x_n,z_n) +L}{d(x_n,z_n)}\leq   \frac{ L}{n}  +\frac{L}{d(x_n,z_n)}\to 0,
\end{align*}
which contradicts that $\Exp_\infty(g\circ f)>0$. 
\end{proof}

\begin{remark}
Recall that the asymptotic dimension of a given metric space is always an upper bound for the asymptotic dimension of any metric space which coarsely embeds in it  (\cite[Proposition 2.2.4 and Theorem 2.2.5]{NowakYuBook}). Hence,  Example \ref{ExampleAsyCoarseLipEqNotPreserveAsyDim}  cannot be improved to hold for some asymptotic coarse Lispchitz equivalence also satisfying that $\Exp(f\circ g)>0$ and $\Exp(g\circ f)>0$.
\end{remark}

The next proposition shows that the asymptotically coarse Lipschitz equivalences can be much more rigid and still not force the spaces to be coarsely equivalent.

\begin{proposition}\label{PropWeakerEqForMetric}
There are metric spaces $X$ and $Y$ which are not coarsely equivalent but such that there is  an asymptotic coarse Lipschitz equivalence $f:X\to Y$ with asymptotic coarse Lipschitz inverse $g:Y\to X$ such that $\Exp(f\circ g)>0$ and $\Exp(g\circ f)>0$. In particular, both $f$ and $g$ are coarse Lipschitz embeddings and $\mathrm{AsyDim}(X)=\mathrm{AsyDim}(Y)$. 
\end{proposition}

\begin{proof}\label{PropWeakerEqForMetric}
Let \[X=\{(x+\log(y+1),y )\mid (x,y)\in [0,\infty)^2\}\] and   \[Y=X\cup \big(\{0\}\times [0,\infty)\big),\] and consider $X$ and $Y$ with their standard metrics inherited from $\R^2$. Let $f:X\to Y$ be the inclusion map and let $g:Y\to X$ be given by 
\[g(x,y)=(x+\log(y+1),y )\]
for all $(x,y)\in Y$. Being the inclusion, $f$ is coarse Lipschitz and, since the logarithm is a coarse Lipschitz function on $[1,\infty)$, so is $g$.

We now notice that $f$ and $g$ are asymptotic coarse Lipschitz inverses of  each other. Since
\begin{align*}
\|g\circ f(x,x')-(x,x')\|=\|(\log(x'+1)),0)\| \leq  \log (\|(x,x')\|+1)
\end{align*}
and $\lim_{s\to \infty}\log(s+1)/s=0$, 
it follows that  $ g\circ f\sim_\infty \mathrm{Id}_X$. Moreover, since 
\begin{align*}
\|g\circ f(x,x')-g\circ f(z,z')\|
\geq \|(x,x')-(z,z')\|-|\log(x'+1)-\log(z'+1)|,
\end{align*}
it also follows that $\Exp_\infty(g\circ f)>0$. Analogously,  we have    that $ f\circ g\sim_\infty \mathrm{Id}_Y$ and $\Exp_\infty(f\circ g)>0$. 

Let us now notice that $X$ and $ Y$ are not coarsely equivalent. Suppose towards a contradiction that there are coarse maps $f:X\to Y$ and $g:Y\to X$ such that $g\circ f\sim \mathrm{Id}_X$ and $f\circ g\sim \mathrm{Id}_Y$. Given $x\in X$ and $r>0$, we let $B(x,r)=\{z\in X\mid d(x,z)\leq r\}$. Since $g: Y\to X$ is a coarse equivalence, we have that for all $s>0$ there is $r>0$ such that \[\|(y,y')-(z,z')\|>s\ \text{ implies }\ \|g(y,y')-g(z,z')\|>r,\]
for all $(y,y'),(z,z')\in Y$. Therefore, we must have  \[\lim_{t\to \infty} d(g(0,t)-g(X))=\infty.\]
So, we can pick   a sequence $(y_n)_n$ of elements in $[0,\infty)$ such that 
\begin{equation}\label{Eq1pppp}
 B(g(0,y_n), n)\cap \big(g(Y)\setminus g(\{0\}\times [0,\infty))\big)=\emptyset
 \end{equation}
for all $n\in\N$. Let \[Z=g(\{0\}\times [0,\infty)) \cup \bigcup_n (B(g(0,y_n), n)\cap X).\]

Let $h=g\restriction \{0\}\times [0,\infty)$; so, $h$ is a coarse embedding. As $g$ is a coarse equivalence, its image is \emph{cobounded}, i.e., $\sup_{x\in X}d(x,g(Y))<\infty$. Therefore,  it follows from \eqref{Eq1pppp} that $h: \{0\}\times [0,\infty)\to Z$ is a coarse equivalence. This is a contradiction since $\{0\}\times [0,\infty)$ has asymptotic dimension $1$,  $Z$ has asymptotic dimension $2$, and coarse equivalences preserve asymptotic dimension (\cite[Theorem 2.2.5]{NowakYuBook}).
\end{proof}

\begin{remark}
We present here another approach to obtain that $X$ and $Y$ are not coarsely equivalent in the previous proposition. We can introduce the following coarse property: a metric space $(X,d)$ is said to be \emph{coarsely connected at infinity} if there is $s>0$ such that for all bounded $A\subset X$ there exists another bounded $B\subset X$ such that any $x,y\in X\setminus B$ can be connected by a discrete path $x_0=x,x_1,\ldots, x_n=y$ such that  $x_i\not\in A$ and   $d(x_i,x_{i-1})\leq s$ for all $i\in \{1,\ldots, n\}$.

It is easy to show that this is a coarse property, i.e., it is preserved by coarse equivalences. But while $X$ in the previous example has it, $Y$ does not. 
\end{remark}

Notice that both restrictions in the previous proposition and in the previous remark only work for metric spaces with asymptotic dimension at least 1 (indeed, the existence of the paths in the remark above implies either that $X$ is   bounded or that $X$ has asymptotic dimension at least 1). Moreover, Example \ref{ExampleAsyCoarseLipEqNotPreserveAsyDim}  falls in the same scenario. This justifies the following problem.

\begin{problem}
Let $X$ and $Y$ be asymptotically coarse Lipschitz equivalent metric spaces with asymptotic dimension zero. Does it follow that $X$ and $Y$ are coarse (Lipschitz) equivalent?  
\end{problem}

\section{Banach spaces and behavior under ultraproducts}\label{ultraproducts}

\subsection{Comparing asymptotic coarse Lipschitz and coarse Lipschitz equivalences}
 
The metric spaces of  main interest in this paper are Banach spaces. From now own, we will study the asymptotic coarse Lipschitz geometry of Banach spaces and show that several of the results present in the coarse Lipschitz theory are still valid under this weaker notion. As a warm up, we start with a simple result about asymptotic coarse Lipschitz equivalences which also satisfy the extra assumption that $\Exp(f\circ g)>0$ and $\Exp(g\circ f)>0$. 
 
 \begin{proposition}\label{PropFinDimStrongerNotion}
 Let $X$ and $Y$ be Banach spaces, and $f:X\to Y$ be an asymptotic coarse Lipschitz equivalence with asymptotic coarse Lipschitz inverse  $g:Y\to X$. Suppose also that   $\Exp(f\circ g)>0$ and $\Exp(g\circ f)>0$.   Then $\dim(X)=\dim(Y)$ and, if     $\dim(X)< \infty$,    $f$ (and $g$) are coarse Lipschitz equivalences.
 \end{proposition}

\begin{proof}
By Proposition \ref{PropAsympEquiImpliesCoarseEmb},  both $f$ and $g$ are  coarse Lipschitz embeddings and $\dim(X)=\dim(Y)$  follows. If moreover $\dim(X)<\infty$ the result follows since any coarse Lipschitz embedding $h:\R^n\to \R^n$ is automatically \emph{cobounded}, i.e., $\sup_{y\in \R^n}d(y, h(\R^n))<\infty$ (see \cite[Exercise 2.27]{Kapovich2014Lectures}), and therefore a coarse Lipschitz equivalence.
\end{proof}

As we will see in Corollary \ref{CorFinDimWeakNotion} below, the first statement of the previous proposition also holds for asymptotic coarse Lipschitz equivalences. However, as seen in Example \ref{ExWeakCoarseEqNotCoarseEqBanSp},  the second does not. We conclude now this very short subsection with two questions.

\begin{problem} 
Let $X$ and $Y$ be two  asymptotically coarse Lipschitz equivalent Banach spaces. Are they necessarily coarse Lipschitz equivalent?
\end{problem}

\begin{problem}
Let $X$ and $Y$ be Banach spaces and let $f:X\to Y$ be as in Proposition \ref{PropFinDimStrongerNotion}. Must  $f$ be a coarse Lipschitz equivalence? If not, does it follow that the existence of such $f$ implies the existence of a coarse Lipschitz equivalence $X\to Y$?
\end{problem}

 \subsection{Ultrapowers and asymptotic coarse Lipschitz equivalences}
 
It is well known that coarse equivalences generate Lipschitz equivalences between ultraproducts. In the next lemma, we show that the same holds for asymptotic coarse Lipschitz equivalences. Recall,  if $X$ is a Banach space and $\cU$ is an ultrafilter on $\N$, we denote the ultraproduct of $X$ with respect to $\cU$ by $X^\N/\cU$. We refer the reader to \cite[Section 11.1]{AlbiacKaltonBook} for the theory of ultraproducts of Banach spaces.  
 
\begin{proposition}\label{lemmaUltrapower}
Let $X$ and $Y$ be Banach spaces and  $\mathcal{U}$ be a nonprincipal  ultrafilter on $\N$. Let $f: X\to Y$ and $g: Y\to X$ be coarse Lipschitz maps such that $ f\circ g\sim_\infty \mathrm{Id}_Y$. Define $F:X^\N/\mathcal{U}\to Y^\N/\mathcal{U}$ and $G:Y^\N/\mathcal{U}\to X^\N/\mathcal{U}$ by
\[F([(x_n)_n])=\Big[\Big(\frac{f(nx_n)}{n}\Big)_n\Big],\ \text{ for all }\ [(x_n)_n]\in X^\N/\mathcal{U}\]
\[G([(y_n)_n])=\Big[\Big(\frac{g(ny_n)}{n}\Big)_n\Big],\ \text{ for all }\ [(y_n)_n]\in Y^\N/\mathcal{U}.\]
Then $F$ and $G$ are Lipschitz and $F \circ G=\mathrm{Id}_{Y^\N/\mathcal{U}}$.

In particular, $F$ is surjective and if, furthermore, we  have that $  g\circ f\sim_\infty \mathrm{Id}_X$, then $F$ is a Lipschitz equivalence with Lipschitz inverse $G$.
\end{proposition}

\begin{proof} Notice that, as $f$ and $g$ are coarse  Lipschitz maps, it follows easily that $F$ and $G$ are well defined and Lipschitz. We only need to show that  $F \circ G=\mathrm{Id}_{Y^\N/\mathcal{U}}$, as the rest of the statement will clearly follow.  It is also obvious that $F \circ G(0)=0$. So let $[(y_n)_n]\in Y^\N/\mathcal{U}$, so that $[(y_n)_n]\neq 0$ and fix $\eps>0$.  

As $f\circ g\sim_\infty \mathrm{Id}_Y$, there exists $r>0 $  such that for all $y \in Y$ with $\|y\|>r$ we have $\|f(g(y))-y\|\le \eps\|y\|$. Since $\lim_{n,\mathcal{U}}\|y_n\|>0$, there exists $A \in \mathcal{U}$ such that $\|f(g(ny_n))-ny_n\|\le \eps\|ny_n\|$, for all $n\in A$. Unfolding definitions, we have that  \[F(G([(y_n)_n]))-[(y_n)_n]=\Big[\Big(\frac{f(g(ny_n))}{n}-y_n\Big)_n\Big]=\Big[\Big(\frac{f(g(ny_n))-ny_n}{n}\Big)_n\Big].\] 
Therefore, 
\[\Big\|F(G([(y_n)_n]))-[(y_n)_n]\Big\|=\lim_{n,\cU}\left\|\frac{f(g(ny_n))-ny_n}{n}\right\|\leq \eps \|[(y_n)_n]\|.\]
As $\eps>0$  was arbitrary, this shows that $F\circ G([(y_n)_n])=[(y_n)_n]$. 
\end{proof}

Unlike coarse Lipschitz equivalences, the range of an asymptotic coarse Lipschitz equivalence does not need to be $\delta$-dense in its codomain for some $\delta>0$ (see Example \ref{ExampleAsyCoarseLipEqNotPreserveAsyDim} or Proposition \ref{PropAsympEquiImpliesCoarseEmb}). The next corollary gives some information on what the range of an asymptotic coarse Lipschitz equivalence must be. 

\begin{corollary}
Let $X$ and $Y$ be Banach spaces, and   $f: X\to Y$ be an asymptotic coarse Lipschitz equivalence. Then $\bigcup_{n\in \N}\frac1nf(X)$ is dense in $Y$.
\end{corollary}

\begin{proof}
Let $F$ be as in Proposition \ref{lemmaUltrapower}. Let $y\in Y$. As $F:X^\N/\mathcal{U}\to Y^\N/\mathcal{U}$ is surjective, there exists $[(x_n)_n]\in X^\N/\mathcal{U}$ such that $F([(x_n)_n])=[(y)_n]\in   Y^\N/\mathcal{U}$. Hence, $\lim_{n,\mathcal{U}}\|\frac1nf(nx_n)-y\|=0$. So, $\bigcup_{n\in \N}\frac1nf(X)$ is dense in $Y$. 
\end{proof}

 \begin{corollary}\label{CorFinDimWeakNotion}
 Let $X$ and $Y$ be asymptotically coarse Lipschitz equivalent Banach spaces. 
 \begin{enumerate}
\item\label{CorFinDimWeakNotion.Item1} If   $\dim(X)< \infty$, then $\dim(X)=\dim(Y)$.
 \item\label{CorFinDimWeakNotion.Item2} If $Y=\ell_2$, then $X$ is isomorphic to $\ell_2$.
\end{enumerate}
 \end{corollary}

\begin{proof}
This follows immediately from Proposition \ref{lemmaUltrapower}. Indeed, if $\dim(X)<\infty$, then all of its ultraproducts also have dimension $\dim(X)$ and the dimension is preserved by Lipschitz equivalences. So, \eqref{CorFinDimWeakNotion.Item1} follows. 

For \eqref{CorFinDimWeakNotion.Item2}, recall that the ultrapower of $\ell_2$ is still a Hilbert space  and that being isomorphic to a Hilbert space is preserved by Lipschitz equivalences (see \cite[Theorem 6.3.1]{Enflo1970Israel}).
\end{proof}

 \section{Asymptotic coarse Lipschitz structure of $\ell_p$, $p\in (2,\infty)$. }\label{SectionAsyCLStruclp}
As shown in Corollary \ref{CorFinDimWeakNotion}, the asymptotic coarse Lipschitz structure of $\ell_2$ completely determines its isomorphic structure. In this subsection, we show that the same holds for $\ell_p$ for any $p \in (2,\infty)$. This is the missing part of Theorem \ref{ThmAsyCLStructurelpINTRO} of Section \ref{SectionIntro} and it is proved below as   Corollary \ref{CorAsyCLStructurelp}. In fact, the same holds for $\ell_{p_1}\oplus\ldots\oplus \ell_{p_n}$ for any $p_1,\ldots,p_n \in (2,\infty)$  (see Theorem \ref{ThmAsyCLStructurelp}). 

Asymptotic uniform smoothness will play an important role in the proof of Theorem \ref{ThmAsyCLStructurelp} below. We refer the reader to Section \ref{SectionIntro} for the precise definitions of \emph{asymptotic uniform smoothness}, abbreviated as AUS, and \emph{$p$-asymptotic uniform smoothness}, abbreviated as $p$-AUS.  For this section, however, it will be enough for the reader to know that, for $p\in (1,\infty)$,   $\ell_p$ is $p$-AUS. Moreover, if $1<p_1<p_2<\ldots< p_n<\infty$, then $\bigoplus_{i=1}^n\ell_{p_i}$ is $p_1$-AUS.

As demonstrated by N. Kalton and L. Randrianarivony \cite{KaltonRandrianarivony2008},  the Hamming  graphs are of great help when studying $p$-AUS spaces. Given an infinite $\M\subset \N$ and $k\in\N$, we let $[\M]^{k}$ be the set of all subsets of $\M$ with exactly $k$ elements and we denote each element of $[\M]^{k}$ as a tuple in an increasing order, i.e., given $\bar n\in [\M]^{k}$, we write $\bar n=(n_1,\ldots,n_k)$ where $n_1<\ldots<n_k$. The \emph{Hamming metric of $[\M]^{k}$} is the metric $d_{\mathrm H}$ given by 
\[d_{\mathrm H}(\bar n,\bar m)=|\{i\in \{1,\ldots, k\}\mid n_i\neq m_i\}|\]
for all $\bar n,\bar m\in [\M]^{k}$. 

The following relation between the Hamming metric and AUSness was found by N. Kalton and L. Randrianarivony:

\begin{lemma}\emph{(}\cite[Theorem 4.2]{KaltonRandrianarivony2008}\emph{).}
Let $p\in (1,\infty)$ and suppose   $Y$ is  a reflexive $p$-AUS Banach space. Then there is $C>0$ such that for all $k\in\N$ and all  Lipschitz maps  $f:[\N]^{k}\to Y$, there is an infinite $\M\subset \N$ such that
\[\mathrm{diam}(f([\M]^{k}))\leq C\Lip(f)k^{1/p}.\]
\label{LemmaThmKR4.2}
\end{lemma}

\begin{theorem}\label{TheoremNol2}
Let $1\leq r< p_1< \ldots< p_n<\infty$. If a Banach space $X$ is asymptotically coarse Lipschitz equivalent to $\bigoplus_{k=1}^n\ell_{p_k}$, then $\ell_r$ does not   coarse Lipschitz embed  into $X$. 
\end{theorem}

\begin{proof}
To simplify   notation, let  $Y=\bigoplus_{k=1}^n\ell_{p_k}$. Assume $X$ is asymptotically coarse Lipschitz equivalent to $Y$ and let $f:X\to Y$ be such an equivalence. Suppose towards a contradiction that there is a coarse Lipschitz embedding $g:\ell_r\to X$. Replacing $g$ by an appropriate translation of itself if necessary, we can assume that $g(0)=0$. Moreover, as $g$ is a coarse Lipschitz embedding, replacing $g$ by $\alpha g(\beta \cdot)$ for appropriate $\alpha,\beta>0$, we can also assume that there is $L>1$ for which 
\[ \|z-z'\|\leq  \|g(z)-g(z')\|\leq L\|z-z'\|\ \text{ for all }\ z,z'\in \ell_r\ \text{ with }\ \|z-z'\|\geq 1.\]
 As $f$ is coarse Lipschitz, replacing $L$ by a larger number if necessary, we can also assume that 
\[\|f(x)-f(x')\|\leq L\|x-x'\|\ \text{ for all }\ x,x'\in X\ \text{ with }\ \|x-x'\|\geq 1.\]

 For each $k\in \N$, define $\varphi_k:[\N]^{k}\to \ell_r$ as 
\[\varphi_k(\bar n)=\sum_{i=1}^ke_{n_i} \ \text{ for all } \bar n\in [\N]^{k}.\]
Notice that $\varphi_k$ is $2^{1/r}$-Lipschitz and that its image is $1$-separated. So, by our choice of $L$, we have that 
\[\Lip(f\circ g\circ \varphi_k)\leq  2^{1/r}L^2\ \text{  for all }\ k\in\N.\] 
Therefore, as $Y$ is $p_{1}$-AUS, it follows from Lemma \ref{LemmaThmKR4.2} that there is $C>0$ such that, for each $k\in\N$, there is an infinite $\M_k\subset \N$ such that  
\begin{equation}\label{Eqkarb1}
\mathrm{diam}(f\circ g\circ \varphi_k([\M_k]^{k}))\leq C 2^{1/r}L^2k^{1/p_1}.
\end{equation}
For a fixed $k\in\N$,  pick  $\bar n,\bar m\in [\M_k]^{k}$ with $n_k<m_1$. Then 
\[\|g(\varphi_k(\bar n))-g(\varphi_k(\bar m))\|\geq 2^{1/r} k^{1/r}\]
and, as we assume that $g(0)=0$, 
\[ \|g(\varphi_k(\bar n))\|\leq L k^{1/r} \ \text{ and } \   \|g(\varphi_k(\bar m))\|\leq L k^{1/r}. \]
Applying Proposition \ref{PropAlmCLInv}  to   $x_0=0$ and $\theta=2^{1/r}/L$, we obtain   $M>1$ (independent on $k$) and $k_0\in \N$ such that, if $k$ was previously chosen larger than $k_0$,
\begin{equation}\label{Eqkarb2}
\|f(g(\varphi(\bar n)))-f(g(\varphi(\bar m)))\|\geq \frac{1}{M}\|g(\varphi(\bar n))-g(\varphi(\bar m))\|\geq \frac{k^{1/r}}{M}
\end{equation}
Then, \eqref{Eqkarb1} and \eqref{Eqkarb2} imply that for all $k\ge k_0$, 
\[\frac{k^{1/r}}{M}\leq C 2^{1/r}L^2k^{1/p_{1}}\]
As $r<p_1$, this gives us a contradiction for large values of $k\in\N$. 
\end{proof}

\begin{remark}\label{RemarkForthcomingPaper} The above statement  could actually be stated in terms  of ``asymptotically coarse Lipschitz embeddings''.  However, the right definition of this type of embedding is not clear: indeed, one could naturally define such embedding as being an asymptotic coarse Lipschitz equivalence between $X$ and a \emph{subset} of $Y$. But this definition has a big fault: it is not clear that the composition of such embeddings would still be an embedding of this sort. For this reason,  we chose to focus only on equivalences in this paper. 

We point out however that the proof of Theorem \ref{TheoremNol2} shows the following stronger result: if $1\leq r< p_1< \ldots< p_n<\infty$, then there is no coarse Lipschitz map $f:\ell_r\to \ell_{p_1}\oplus \ldots \oplus \ell_{p_n}$ which also satisfies the conclusion of Proposition \ref{PropAlmCLInv}. We will study maps  between Banach spaces which are coarse Lipschitz and satisfy the conclusion of Proposition \ref{PropAlmCLInv} (as well as modifications of it) in  a forthcoming paper  (see \cite{BragaLancienAsyEmb}).
\end{remark}

\begin{remark}\label{RemarkMidPoint}
The reader familiar with the nonlinear theory of Banach spaces knows that, for the classic coarse Lipschitz equivalences, Theorem \ref{TheoremNol2} also holds if $r$ is larger than all $p_i$'s. More precisely, $\ell_r$ does not coarse Lipschitz embeds into $\ell_{p_1}\oplus\ldots\oplus\ell_{p_n}$ for all $1\leq p_1<\ldots<p_n<r$. 
The proof of this result does not use Hamming graphs but instead uses what is called the ``approximate midpoints technique'' (see \cite[Proposition 3.1]{KaltonRandrianarivony2008}). This however does not hold in our settings. Indeed, we show in our forthcoming paper  (see \cite{BragaLancienAsyEmb}) that, for $p>q$,  $\ell_p$ can be mapped into $\ell_q$ by a coarse Lipschitz maps which satisfies the conclusion of  Proposition 2.3 (spoiler: the  Mazur map does that). See Remark \ref{RemarkMidPointEquiv} below.
\end{remark}

We can now state and prove our extension of \cite[Theorem 2.2]{JohnsonLindenstraussSchechtman1996GAFA} and  \cite[Theorem 5.3]{KaltonRandrianarivony2008} to asymptotically coarse Lipschitz equivalences.

\begin{theorem} Let $2< p_1< \ldots< p_n<\infty$. If a Banach space $X$ is asymptotically coarse Lipschitz equivalent to $\bigoplus_{k=1}^n\ell_{p_k}$, then $X$ is linearly isomorphic to $\bigoplus_{k=1}^n\ell_{p_k}$.\label{ThmAsyCLStructurelp}
\end{theorem}

\begin{proof}
The proof follows by induction on $n\in\N$. Indeed, suppose $X$ is asymptotically coarse Lipschitz equivalent to $\ell_{p_1}$. Then, by Proposition \ref{lemmaUltrapower}, we have that $X^\N/{\cU}$ is Lipschitz equivalent to $Y^\N/{\cU}$, where $\cU$ is any nonprincipal ultrafilter on $\N$. In particular, it follows from classic results that  $X$ is isomorphic to a complemented subspace of $L_{p_1}$ (this follows for instance from \cite[Theorems 1.3 and 2.3, and Proposition 2.1]{HeinrichMankiewicz1982}). By Theorem \ref{TheoremNol2},   $X$ does not contain an isomorphic copy of $\ell_2$. Hence,   $X$ must be isomorphic to a complemented subspace of  $\ell_{p_1}$ (see the main result in \cite{Johnson1976JLMS}), which in turn is either isomorphic to $\ell_{p_1}$ itself or finite dimensional (\cite[Theorem 1]{Pelczynski1960Studia}). Since $X$ cannot be finite dimensional, the case $n=1$ follows. 

Let $n>1$ and suppose the result follows for any $m<n$. If $X$ is asymptotically coarse Lipschitz equivalent to $\bigoplus_{k=1}^n\ell_{p_k}$,   arguing exactly as above, we obtain that  $X$ is isomorphic to a complemented subspace of $\bigoplus_{k=1}^nL_{p_k}$. Since $X$ does not contain $\ell_2$  (Theorem \ref{TheoremNol2}), this implies that $X$ must be isomorphic to a complemented subspace of $\bigoplus_{k=1}^n\ell_{p_k}$ (this follows again from the main result of \cite{Johnson1976JLMS})). As this is the direct sum of completely incomparable Banach spaces, this implies that $X$ is isomorphic to $\bigoplus_{k=1}^nE_k$, where each $E_k$ is a complemented subspace of $\ell_{p_i}$ (\cite[Corollary 3.7]{EdelsteinWojtaszczyk1976Studia}); so, each $E_k$ is either finite dimensional or  isomorphic to $\ell_{p_k}$ itself. In order to conclude, we need to show that each $E_i$ is isomorphic to $\ell_{p_i}$. Otherwise $X$ would be isomorphic to $Y=\bigoplus_{i \in I}\ell_{p_i}$, for some  $I \subsetneq \{p_1,\ldots,p_n\}$ and thus, $\bigoplus_{k=1}^n\ell_{p_k}$ would be asymptotically coarse Lipschitz equivalent to $Y$, contradicting our inductive assumption. 
\end{proof}

The following is an immediate consequence of Corollary \ref{CorFinDimWeakNotion} and  Theorem  \ref{ThmAsyCLStructurelp}.

\begin{corollary} [Theorem \ref{ThmAsyCLStructurelpINTRO} in Section \ref{SectionIntro}]
Let $p\in [2,\infty)$. If a Banach space is asymptotically coarse Lipschitz equivalent to $\ell_p$, then $ X$ is linearly isomorphic to $\ell_p$. \label{CorAsyCLStructurelp}
\end{corollary}

\begin{remark}\label{RemarkMidPointEquiv}
We point out that we do not know if Corollary \ref{CorAsyCLStructurelp} is valid for $p$'s in the range $[1,2)$. We only know that, if true, the proof would have to use different ideas than the ones from coarse Lipschitz equivalences (see Remark \ref{RemarkMidPoint} above for more details on that).
\end{remark}

\section{The Gorelik principle and applications to the asymptotic structure}\label{asymptoticstructure}
In this section, we establish a version of the Gorelik principle for asymptotically coarse Lipschitz equivalences. Then, we apply it to  extend to asymptotically coarse Lipschitz equivalences a few results on the stability of asymptotic smoothness properties of Banach spaces under nonlinear equivalences. The first results in this direction can be found in \cite{GodefroyKaltonLancien2001Trans} where it is shown that being $p$-AUSable is stable under Lipschitz equivalences and that being $p'$-AUSable for all $p'<p$ is stable under uniform homeomorphism. However, Kalton proved that being $p$-AUSable is not stable under uniform homeomorphisms or coarse Lipschitz equivalences (\cite[Theorem 5.4 and Remark in page 170]{Kalton2013Examples}). In order to describe our results precisely, we 
 will first introduce  a few other asymptotic properties.

\subsection{Relevant asymptotic properties} \label{SubsectionAsyProp}

It will be useful to define a modulus that is dual to  the modulus of asymptotic uniform smoothness of a Banach space $X$, namely, the modulus of weak$^*$-asymptotic uniform convexity of $X^*$. Firstly, denote the set of all weak$^*$-closed subspaces of $X^*$ with finite codimension by  $\co^*(X^*)$. We can then define the  \emph{modulus of weak$^*$-asymptotic uniform convexity of $X^*$}  by letting
\[\bar\theta_X(\tau)=\inf_{x^*\in  S_{X^*}}\sup_{Y\in \co^*(X^*)}\inf_{y^*\in S_{Y}}\|x^*+\tau y^*\|-1
\]
for all $\tau\geq 0$. We say that $X^*$ is  \emph{weak$^*$-asymptotically uniformly convex}, abbreviated as \emph{w$^*$-AUC}, if 
\[\bar\theta_X(\tau)>0\ \text{ for all }\  \tau >0.\]
It is easy to notice (and we will use this in the proof of Theorem \ref{ThmTheta*Pres}) that $\bar\theta_X(\tau)/\tau$ is increasing. These moduli are related to each other in the sense that $X$ is AUS if and only if $X^*$ is $w^*$-AUC. This follows from the following precise quantitative result that we shall also need.

\begin{proposition}\emph{(}\cite[Proposition 2.1]{DilworthKutzarovaLancienRandrianarivony2017}\emph{).}
Let $X$ be a Banach space and $\tau,\sigma\in (0,1)$. 
\begin{enumerate}
\item If $\bar\rho_X(\sigma)<\sigma\tau$, then $\bar\theta_X(6\tau)\geq \sigma\tau $
\item If $\bar \theta_X(\tau)>\sigma\tau$, then $\bar\rho_X(\sigma)\leq \sigma\tau$.
\end{enumerate} \label{PropositionRelationRhoTheta}
\end{proposition}

This can be rephrased in terms of Young's duality between $\bar\rho_X$ and $\bar\theta_X$. Recall that, given a continuous function $f:[0,1]\to \R$, its \emph{dual Young function} $f^*:[0,1]\to \R$ is defined as
\[f^*(s)=\sup \{st-f(t)\mid t\in [0,1]\}\]
for all $s\in [0,1]$. Notice that, if $f,g:[0,1]\to \R$ are so that $f(t/C)\leq g(t)$ for all $t\in [0,1]$ and some $C>0$, then $g^*(t/C)\leq f^*(t) \text{ for all }\ t\in [0,1]$. The following is then a simple consequence of Proposition \ref{PropositionRelationRhoTheta}.

\begin{proposition}\emph{(}\cite[Corollary 6.2]{DaletLancien2017NWEJM}\emph{).}
Given a Banach space $X$, we have that 
\[\bar \rho_X(s/2)\leq (\bar\theta_X)^*(s)\ \text{ and } \ (\bar\theta_X)^*(s/6)\leq \bar \rho_X(s)\]
 for all $ s\in [0,1]$.\label{PropositionPhoThetaEquiv}
\end{proposition}

We now turn to the asymptotic isomorphic properties that we shall consider.  First, if $\cD$ is a set and $k\in\N$, we write  $\cD^{\leq k}=\bigcup_{i=1}^k\cD^i$ and $\cD^{<\omega}=\bigcup_{i=1}^\infty\cD^i$.  Given a Banach space $X$ and a  weak neighborhood basis of $0\in X$, say $\cD$, we call a family $(x_{\bar n})_{\bar n\in \cD^{\leq k}}$ (respectively $(x_{\bar n})_{\bar n\in \cD^{<\omega}}$) a \emph{weakly null tree} if for each $\bar n\in \{\emptyset\}\cup \cD^{\leq k-1}$ (respectively for each $\bar n\in \{\emptyset\}\cup \cD^{<\omega}$) the net $(x_{(\bar n,n_k)})_{n_k\in \cD}$ is weakly null; where here we consider $\cD$ as a directed set with the usual reverse inclusion order.

\begin{definition}\label{DefinitionTp}
Let $p\in (1,\infty]$ and $X$ be a Banach space. We say that $X$ has property $\textsf{{T}}_p$ if there is $c>0$ such that for all weak neighborhood basis $\cD$ of $0\in X$,  and all weakly null trees $(x_{\bar n})_{\bar n\in \cD^{<\omega}}$ in $B_X$, there is $\bar m=(m_1,\ldots,m_k,\ldots)\in \cD^{\N}$ such that 
\[\Big\|\sum_{i=1}^\infty a_i x_{(m_1,\ldots, m_i)}\Big\|\leq c\|a\|_{p}\]
for all $a=(a_i)_{i=1}^\infty \in \ell_p$. 
\end{definition}

It was proved by R.M. Causey in  that $X$ is $p$-AUSable if and only if $X$ has property $\textsf{{T}}_p$ (\cite[Theorem 1.1(i)]{CauseyPositivity2018}). We shall concentrate on two slightly weaker properties. 

\begin{definition}\label{DefinitionAp}
Let $p\in (1,\infty]$ and $X$ be a Banach space. We say that $X$ has property $\textsf{{A}}_p$ if there is $c>0$ such that for all weak neighborhood basis $\cD$ of $0\in X$, all $k\in\N$, and all weakly null trees $(x_{\bar n})_{\bar n\in \cD^{\leq k}}$ in $B_X$, there is $\bar m=(m_1,\ldots,m_k)\in \cD^{ k}$ such that 
\[\Big\|\sum_{i=1}^ka_i x_{(m_1,\ldots, m_i)}\Big\|\leq c\|a\|_{p}\]
for all $a=(a_1,\ldots, a_k)\in \ell_p^k$. 
\end{definition}

\begin{definition}\label{DefinitionAp}
Let $p\in (1,\infty]$ and $X$ be a Banach space. We say that $X$ has property $\textsf{{N}}_p$ if there is $c>0$ such that for all weak neighborhood basis $\cD$ of $0\in X$, all $k\in\N$, and all weakly null trees $(x_{\bar n})_{\bar n\in \cD^{\leq k}}$ in $B_X$, there is $\bar m=(m_1,\ldots,m_k)\in \cD^{ k}$ such that 
\[\Big\|\sum_{i=1}^kx_{(m_1,\ldots, m_i)}\Big\|\leq ck^{1/p}\]
(if $p=\infty$, we use the convention that $1/\infty=0$).
\end{definition}

The next theorem gathers the relations between the classes $\textsf{T}_p$, $\textsf{A}_p$, and $\textsf{N}_p$.

\begin{theorem}\emph{(}\cite[Theorem 1.1]{Causey2018ThreeAndHalf}\emph{).} Let $p\in (1,\infty)$. Then
$$\textsf{\emph{T}}_p \subsetneq\textsf{\emph{A}}_p \subsetneq \textsf{\emph{N}}_p \subsetneq \bigcap_{p'<p}\textsf{\emph{T}}_{p'}\ \ \text{and}\ \ \textsf{\emph{T}}_\infty \subsetneq\textsf{\emph{A}}_\infty = \textsf{\emph{N}}_\infty \subsetneq \bigcap_{p'<\infty}\textsf{\emph{T}}_{p'}$$\label{inclusions} 
\end{theorem}

We can now state our most precise result, which is the main result of this section and will imply Theorem \ref{ThmAUSPreservationINTRO}. 

\begin{theorem}\label{precise stability} Let $p\in (1,\infty]$. Then properties $\textsf{\emph{A}}_p$ and $\textsf{\emph{N}}_p$ are stable under asymptotically coarse Lipschitz equivalences. 
\end{theorem}

\subsection{The Gorelik principle}

Before we proceed with the proof of Theorem \ref{precise stability}, we need to establish  a variant of the  
 Gorelik principle that is valid for such equivalences. This subsection takes care of this. This principle was initially manufactured for uniform homeomorphisms and named after Gorelik's pioneer work (\cite{Gorelik1994}). Here is our version for continuous asymptotically coarse Lipschitz equivalences.

\begin{proposition}\label{PropGorelik}
Let $X$ and $Y$ be Banach spaces and assume that there exist  a continuous asymptotic coarse Lipschitz equivalence $f:X\to Y$ with continuous asymptotic coarse Lipschitz inverse $g:Y\to X$. Assume also that $f(0)=0$ and $g(0)=0$. Then there exists $M\geq 1$ such that for all $\eps>0$, there exists $t_0>0$  such that for any finite codimensional subspace $X_0$ of $X$ and  any $t>t_0$, there is a compact subset  $K\subset Y$ such that
\[\frac{t}{M }B_Y \subset K+\eps tB_Y+f(2tB_{X_0}).\]
\end{proposition}

The following lemma will be needed in the proof of Proposition \ref{PropGorelik}.

\begin{lemma}
\emph{(}\cite[Claim (i) of Theorem 10.12]{BenyaminiLindenstraussBook}\emph{).}
Let $X$ be a Banach space, $X_0\subset X$ be a subspace with finite codimension, and $t>0$. There is a compact $A\subset tB_X$ satisfying the following: if  $\varphi:A\to X$ is a continuous map such that $\|\varphi(a)-a\|\leq t/2$ for all $a\in A$, then  $\varphi(A)\cap X_0\neq \emptyset$. \label{LemmaClaim}
\end{lemma}

\begin{proof}[Proof of Proposition \ref{PropGorelik}]
Let $L\geq 1$ be such that $\omega_f(t)\leq Lt+L$ and $\omega_g(t)\leq Lt+L$, for all $t\in [0,\infty)$. Let us show the proposition holds for $M=12 L$. For that, fix $\eps>0$ and let $\delta=\min\{\frac{1}{12L^2},\frac{\eps}{4L^2}\}$. Since  $g\circ f\sim_\infty \mathrm{Id}_X$ and $f\circ g\sim_\infty \mathrm{Id}_Y$, there is   $t_0>0$  such that \[\|g(f(x))-x\|\leq \delta \|x\|\text{ and } \|f(g(y))-y\|\leq \delta \|y \| \]
for all $x\in X$ and $y\in Y$ with  $\|x\|\geq \delta t_0$ and $\|y\|\geq \delta t_0$. Furthermore, we assume that $\delta t_0\geq 12L^2$.

Let $t\geq t_0$ and $X_0\subset X$ be a subspace with finite codimension. By Lemma \ref{LemmaClaim},   there is a compact subset $A\subset t B_X$ such  that if $\varphi:A\to X$ is a continuous map satisfying   $\|\varphi(a)-a\|\leq \frac{1}{2}t$, for all $a\in A$, then   $\varphi(A)\cap X_0\neq \emptyset$. 

Fix $y\in \frac{t}{12L} B_Y$ and define a map $\varphi:A\to X$ by  $\varphi(a)=g(y+f(a))$, for all $a\in A$. As $g$ is continuous, so is $\varphi$. If $a\in A\cap \delta tB_X$, we have
\begin{align*}
\|\varphi(a)-a\|&\leq\|g(y+f(a))-g(f(a))\|+\|g(f(a))\|+\|a\|\\
&\leq L\|y\|+L+L(L\delta t+L)+L +\delta t\\
& \leq \frac{t}{2}.
\end{align*}
On the other hand, if $a\in A\setminus \delta t B_Y$, as $t\geq t_0$, our choice of $t_0$ gives us that     
\begin{align*}
\|\varphi(a)-a\|&\leq\|g(y+f(a))-g(f(a))\|+\|g(f(a))-a\|\\
&\leq L\|y\|+L+\delta \|a\|\\
&\leq \frac{t}{2}.
\end{align*} 
So, $\|\varphi(a)-a\|\leq \frac{t}{2}$, for all $a\in A$. Hence, by our choice of $A$, there exists $a_y\in A$ such that $\varphi(a_y)\in X_0$. Since, $\|a_y\|\leq t$ and $\|\varphi(a_y)-a_y\|\leq \frac{t}{2}$, we have that $\varphi(a_y)\in 2tB_{X_0}$.

Now let us notice that 
\[\|f(g(y+f(a_y)))-(y+f(a_y))\|\leq \eps t.\]
Indeed, if $y+f(a_y)\in \delta t B_Y$, then, since $f(0)=g(0)=0$, we have that
\begin{align*}
\|f(g(y+f(a_y)))-(y+f(a_y))\|&	\leq \|f(g(y+f(a_y)))\|+\|y+f(a_y)\|\\
&\leq L(L\delta t+L)+L+ \delta t\\
&\leq \eps t.
\end{align*}
On the other hand, if $y+f(a_y)\not\in \delta t B_Y$, our choice of $t_0$ and the fact that $f(0)=0$ gives that  
\begin{align*}
\|f(g(y+f(a_y)))-(y+f(a_y))\|&\leq \delta \|y+f(a_y)\|\\
&\leq \delta\|y\|+ \delta (Lt+L)\\
&\leq \eps t.
\end{align*}
Therefore, we conclude that 
\[y\in K+\eps t B_Y+f(2tB_{X_0}),\]
where $K=-f(A)$.  As $f$ is continuous, $K$ is compact.
\end{proof}

We now explain how to drop the continuity assumptions in Proposition \ref{PropGorelik} in order to get the following version of the Gorelik principle for asymptotically coarse Lipschitz equivalences.

\begin{theorem}[Gorelik principle for asymptotic coarse Lipschitz equivalences]
Let $X$ and $Y$ be asymptotically coarse Lipschitz equivalent Banach spaces. Then there exist an asymptotic coarse Lipschitz equivalence   $f: X\to Y$  satisfying the following:  there exists $M\geq 1$ such that for all $\eps>0$, there is   $t_0>0$ such that for any   finite codimensional subspace $X_0$ of $X$ and any   $t>t_0$, there is a compact subset  $K\subset Y$ such that
\[\frac{t}{M }B_Y \subset K+\eps tB_Y+f(2tB_{X_0}).\]\label{ThmGorelik}
\end{theorem}

\begin{proof}
Let $f:X\to Y$ and $g:Y\to X$ witness that $X$ and $Y$ are asymptotically coarse Lipschitz equivalent.  Without loss of generality, we can assume that $f(0)=0 $ and $g(0)=0$. By  \cite[Theorem 1.4]{Braga2017JFA}, there exist continuous maps $\tilde{f}:X\to Y$ and $\tilde{g}:Y\to X$ such that \[\sup_{x\in X}\|f(x)-\tilde{f}(x)\|<\infty\ \text{  and }\ \sup_{x\in X}\|g(x)-\tilde{g}(x)\|<\infty.\]  Without loss of generality, we assume that  $\tilde{f}(0)=0$ and $\tilde{g}(0)=0$. Since $f$ is close to $\tilde f$ and $g$ is close to $ \tilde g$, it immediately follows that $\tilde f$ is also  an asymptotic coarse Lipschitz equivalence   with asymptotic coarse Lipschitz  inverse $\tilde g$. The result now follows from Proposition \ref{PropGorelik} applied to $\tilde{f}$ and $\tilde{g}$.
\end{proof}

We finish this subsection with a remark about Theorem \ref{ThmGorelik} above. For this, recall: if $L\geq 1$, $Y$ is a vector space, and $\|\cdot\|$ and  $|\cdot|$ are norms on $Y$ such that
\[\frac{1}{L}\|y\|\leq |y|\leq L \|y\|\ \text{ for all }\ y\in Y,\]
we say that $\|\cdot\|$ and $|\cdot|$ are \emph{$L$-equivalent} and write $\|\cdot\|\sim_L|\cdot|$.

\begin{remark}\label{RemarkConstantGorelik}
Notice that the constant $M$ obtained in the proof of Proposition \ref{PropGorelik} equals $12L$, where $L$ is simply a number such that $\omega_f(t)\leq Lt+L$ and $\omega_g(t)\leq Lt+L$ for all $t\in [0,\infty)$. This implies in particular that if $L'\geq 1$ and  $|\cdot|$ is a norm on $X$ which is  $L'$-equivalent to the original norm of $X$, then Theorem \ref{ThmGorelik}  holds for the Banach spaces $Y$ and $(X,|\cdot|)$ with  $M=12LL'$. This will be useful below.
\end{remark}

\subsection{The technical renorming result} We present in this subsection the key renorming result for Banach spaces that are asymptotically coarse Lipschitz equivalent to an AUS Banach space. Its proof follows ideas in \cite[Theorem 5.3]{GodefroyKaltonLancien2001Trans}  (see also \cite[Theorem 3.12]{GodefroyLancienZizler2014Rocky} and \cite{DaletLancien2017NWEJM}). 

\begin{theorem}\label{ThmTheta*Pres}
Let $X$ be an AUS Banach space and let  $Y$ be a Banach space   asymptotically coarse Lipschitz equivalent to $X$. There are $L,C\geq 1$ such that  for all $\delta \in (0,1)$ there is a norm $|\cdot|$ on $Y$   such that $|\cdot|\sim_L\|\cdot\|_Y$ and 
\[\bar\theta_{(Y,|\cdot|)}(\tau)\geq \bar\theta_X(\tau/C)-\delta\ \text{ for all }\ \tau\in (0,1).\]
\end{theorem}

\begin{proof} 
 Let $f:X\to Y$ be the asymptotic coarse Lipschitz  equivalence given by Theorem  \ref{ThmGorelik},  $g:Y\to X$ be an asymptotic coarse Lipschitz inverse of $f$, and $M>0$ be given by Theorem  \ref{ThmGorelik} for $f$. Moreover, let $t_0>0$ be given by Theorem  \ref{ThmGorelik} for $\eps=1/(8M)$, i.e.,   for all   finite codimensional subspaces $X_0\subset X$ and all  $t>t_0$, there is a compact subset  $K\subset Y$ such that
\[\frac{t}{M }B_Y \subset K+\frac{t}{8M} B_Y+f(2tB_{X_0}).\]
As $f$ and $g$ are coarse Lipschitz, we can pick  $L>0$ large enough so that 
\begin{itemize}
\item $\|f(x)-f(x')\|\leq L\max \{\|x-x'\|,1\}$ for all $x,x'\in X$, and 
\item $\|g(y)-g(y')\|\leq L \max \{\|y-y'\|,1\}$ for all $y,y'\in Y$. 
\end{itemize}

For each $k\in\N$, we define an equivalent norm $|\cdot|_k$ on $Y^*$ by letting 
\[|y^*|_k=\sup\left\{\frac{|y^*(f(x)-f(x'))|}{\|x-x'\|}\mid x,x'\in X \text{ with }\|x-x'\|\geq 2^k\right\}\]
Clearly, $|y^*|_{k+1}\leq |y^*|_k$ for all $y^*\in Y^*$ and all $k\in\N$. The next claim shows that this is indeed an equivalent norm on $Y^*$.

\begin{lemma}\label{ClaimEquivNorm}
For all $k\in\N$ and all $y^*\in Y^*$,   we have that \[\frac{1}{L}\|y^*\|\leq |y^*|_k\leq L\|y^*\|.\]
\end{lemma}

\begin{proof}
Fix $k\in\N$ and $y^*\in Y^*$. The inequality $ |y^*|_k\leq L\|y^*\|$ is immediate, so, we only prove the lower bound for $|y^*|_k$.
Fix a sequence $(y_n)_n$ in $S_Y$ such that $\|y^*\|=\lim_n|y^*(y_n)|$. By our choice of $L$, we have that 
\begin{equation}\label{Eq1}
\|g(ny_n)-g(0)\|\leq L\|ny_n\|
\end{equation}
for all $n\in\N$ large enough. It follows that for all $n\in \N$,
\begin{align*}
&\frac{|y^*(f(g(ny_n))-f(g(0)))|}{\|g(ny_n)-g(0)\|}\\
&\geq \frac{1}{L}\frac{|y^*(f(g(ny_n))-f(g(0)))|}{\|ny_n\|}\\
&\geq  \frac{1}{L}\left(\frac{|y^*(ny_n)|}{\|ny_n\|}- \frac{|y^*(f(g(ny_n))-ny_n)|}{\|ny_n\|}-\frac{|y^*(f(g(0)))|}{\|ny_n\|}\right).
\end{align*}
As $\lim_{n\to \infty}\|ny_n\|=\infty$ and, as $f$ and $g$ are asymptotic coarse Lipschitz inverses of each other,  we have   
\[\lim_{n\to\infty} \frac{\|f(g(ny_n))-ny_n\|}{\|ny_n\|}=0.\]
Therefore,  we conclude that  
\[\lim_{n\to \infty}\frac{|y^*(f(g(ny_n))-f(g(0)))|}{\|g(ny_n)-g(0)\|}\ge \frac{1}{L}\|y^*\|\]
and the lemma follows.
\end{proof}

Note that $|\cdot|_k$ is clearly weak$^*$ lower semi-continuous and is the dual norm of an equivalent norm  on $Y$ whose closed unit ball is the closed convex hull of 
$$\left\{\frac{f(x)-f(x')}{\|x-x'\|}\mid x,x'\in X \text{ with }\|x-x'\|\geq 2^k\right\}.$$

\begin{lemma}\label{ClaimInTec}
Let $\delta \in (0,1)$. There is $C>0$ such that for all $k\in\N$ with $2^kC\delta^{-1}\bar\theta_X(\delta/C)>2t_0$, the following holds: for all  $\tau\in (\delta,1)$ and all $y^*\in LB_Y$, there is  $Z\in \co^*(Y^*)$  such that 
\[|y^*+z^*|_k\geq 2|y^*|_{k+1}-|y^*|_k+\bar\theta_X(\tau/C)\]
for all $z^*\in Z$ with $\|z^*\|\geq \tau/L$.
\end{lemma}

\begin{proof}
The lemma above is only nontrivial for elements $z^*$ with a moderately small norm, i.e.,   it is enough to show that  there are $C>0$ such that for all $k\in\N$ as above, all  $\tau\in (\delta,1)$,  and all $y^*\in LB_Y$ there is  $Z\in \co^*(Y^*)$  such that 
\[|y^*+z^*|_k\geq 2|y^*|_{k+1}-|y^*|_k+\bar\theta_X(\tau/C)\]
for all  $z^*\in Z$ with $\|z^*\|\in ( \tau/L,s)$; where $s>\tau/L$ is a number large enough depending on $L$ only. For now on, fix such $s$  and pick  some
\begin{equation}\label{EqC}
C>16ML+64 ML^3+32s ML^2.
\end{equation}
For the remainder of the proof, we show that $C$  has the required properties. For that, fix  $k\in\N$ as required,  $\tau\in (\delta,1)$   and $y^*\in LB_Y$. We also  fix $\gamma>0$  and some positive $\beta<\bar\theta_X(\tau/C)$ throughout the proof (this is possible because $X$ is AUS and therefore $X^*$ is $w^*$-AUC). Moreover, since $\bar\theta_X(t)/ t$ is increasing, our choice of $k$ allows us to   assume that $\beta$ also satisfies  $2^k\beta C\tau^{-1}> 2t_0$.

By the definition of the norm $|\cdot|_{k+1}$, we can pick $x,x'\in X$ with $\|x-x'\|\geq 2^{k+1}$ such that 
\[y^*(f(x)-f(x'))\geq (1-\gamma)\|x-x'\||y^*|_{k+1}.\]
In order to simplify notation, notice that, replacing $f$ by $f(\cdot - x_1)+x_2$, for appropriate $x_1,x_2\in X$, we can assume  that $x=-x'$ and $f(x)=-f(x')$.   In particular, $\|x\|\geq 2^k$ and 
\begin{equation}\label{Eq3}
y^*(f(x))=\frac{1}{2}y^*(f(x)-f(x'))\geq (1-\gamma)|y^*|_{k+1}\|x\|.
\end{equation}

Letting  $\sigma=\beta C /\tau$, item $(2)$ of Proposition \ref{PropositionRelationRhoTheta} implies that  $\bar\rho_X(\sigma)\leq \beta$. Therefore, there is a finite codimensional $X_0\subset X$ such that 
\begin{equation}\label{Eq6}\|x+z\|\leq (1+2\beta)\|x\|\ \text{ for all }\ z\in \sigma\|x\|B_{X_0}.
\end{equation}
Replacing $X_0$ by a smaller finite codimensional subspace, we can also assume without loss of generality that 
\begin{equation}\label{Eq9}
\|x+z\|\geq \|x\|\geq 2^k \ \text{ for all } \ z\in \sigma\|x\|  B_{X_0}.
\end{equation}
 
As $\|x\|\geq 2^k$, it follows from our choice of $\beta$ that  $\sigma\|x\|>2t_0$. Therefore, our choice of $t_0$ implies that  there is a compact $K\subset Y$ such that    
\begin{equation}\label{Eq2}
\frac{\sigma \|x\|}{2M}B_Y\subset K+ \frac{\sigma \|x\|}{16M} B_Y+f(\sigma\|x\|B_{X_0}).
\end{equation}
Since $K$ is compact, there is $Z\in \co^*(Y^*)$ such that 
\[z^*(f(x))=0\ \text{ and }\ |z^*(y)|\leq \frac{\sigma \|x\|\|z^*\|}{8M}\]
for all $y\in K$ and all $z^*\in sB_Z$.

We will now proceed to show that     $Z$ chosen above has the desired properties. For that, fix $z^*\in Z$ with $\|z^*\|\in [\tau/L,s)$ and let us estimate $|y^*+z^*|_{k}$ from below. For that, let  $z\in \frac{\sigma \|x\|}{2M} S_Y$ be such that \[z^*(z)\geq  \frac{\sigma \|x\|\|z^*\|}{4M}.\] Then, using \eqref{Eq2} for $-z$, we obtain   $w\in \sigma\|x\| B_{X_0}$ such that 
\begin{equation}\label{Eq5}
z^*(-f(w))\geq \frac{\sigma \|x\|\|z^*\|}{16M}\geq \frac{\beta C\|x\|}{16ML}  .
\end{equation}

Since $x=-x'$, $f(x)=-f(x')$, $\|w-x'\|\geq 2^k$ (this follows from  \eqref{Eq9}), and $\|x'-w\|\le (1+2\beta)\|x\|$ \eqref{Eq6}, we have that 
\[y^*(f(w)+f(x))=y^*(f(w)-f(x'))\leq (1+2\beta)|y^*|_k\|x\|.\]
The inequality above and  \eqref{Eq3} give us that
\begin{equation}
\label{Eq4}
y^*(f(w))\leq \big((1+2\beta)|y^*|_k-(1-\gamma)|y^*|_{k+1}\big)\|x\|.
\end{equation}
Hence, as $z^*(f(x))=0$, it follows from \eqref{Eq3}, \eqref{Eq5}, and \eqref{Eq4}  that
\begin{align*}
(y^* & +z^*)(f(x)-f(w))
\geq  \Big((2-2\gamma)|y^*|_{k+1} -|y^*|_k - 
2\beta|y^*|_k 
+\frac{  \beta C}{16ML}\Big)\|x\|
\end{align*}
Since, by the definition of $|\cdot|_k$, \eqref{Eq6}, and \eqref{Eq9}, we have 
\[ (y^*+z^*)(f(x)-f(w))\leq (1+2\beta) |y^*+z^*|_{k}\|x\|,\] 
and since $\gamma>0$ was arbitrary,  we conclude that 
\begin{align}\label{Eq11}
(1+2\beta) |y^*+z^*|_{k} \geq & 2|y^*|_{k+1} -|y^*|_k 
-2\beta|y^*|_k+\frac{ \beta C}{16ML}.
\end{align}

As $\|y^*\|\leq L$ and $\|z^*\|\leq s$,  we have that   $|y^*|_k\leq L^2$ and $|z^*|_k\leq s L$. Hence,   \eqref{Eq11}  gives us 
\begin{align*}
|y^*+z^*|_{k}\geq  2|y^*|_{k+1} -|y^*|_k 
+ \frac{   C-64 ML^3-32s ML^2}{16ML}\beta.
\end{align*}
By our choice of $C$ (see \eqref{EqC}) and as $\beta$ is any positive number  smaller than $\bar\theta_X(\tau/C)$, the lemma is proven. 
\end{proof}

We now conclude the proof of the theorem. For that, fix $\delta \in (0,1)$ and let $C\geq 1$   be given by Lemma \ref{ClaimInTec}. Since $\bar\theta_X(\tau/C)\leq \tau/C\leq \tau$ for all $\tau\in (0,1)$, we only need to find a renorming of $Y$ such that 
\[\bar\theta_Y(\tau)\geq \bar\theta_X(\tau/C)-\delta\ \text{ for all }\ \tau\in (\delta ,1).\]

Fix $k_0\in\N$ with $2^{k_0}C\delta^{-1}\bar\theta_X(\delta/C)>2t_0$ and   $N\in\N$ with $2L^2<\delta N$.   Define a norm $|\cdot|$ on $Y^*$ by letting 
\[|y^*|=\frac{1}{N}\sum_{k=k_0+1}^{k_0+N}|y^*|_{k}\ \text{ for all }\ y^*\in Y^*.\]
Clearly, $|\cdot|$ is weak$^*$ lower semi-continuous and, by Lemma \ref{ClaimEquivNorm}, it is equivalent to the original norm of  $Y^*$; in fact,
\[\frac{1}{L}\|y^*\|\leq |y^*|\leq L\|y^*\|\]
for all $y\in Y^*$.  Hence, $|\cdot|$  is a dual norm, i.e., there is a norm on $Y$ equivalent to $Y$'s original norm whose dual norm is $|\cdot|$; moreover, this norm is  $L$-equivalent to $Y$'s original norm. By abuse of notation, we also denote this norm on $Y$ by $|\cdot|$.

In order to conclude, let   us estimate $\bar\theta_{(Y,|\cdot|)}(\tau)$ from below. For that, pick $y^*\in Y^*$ with $|y^*|=1$; so $\|y^*\|\leq L$. Let   $Z\in \co^*(Y^*)$ be given by Claim \ref{ClaimInTec} for $\tau$ and $y^*$ so that 
\[|y^*+z^*|_k\geq 2|y^*|_{k+1}-|y^*|_k+\bar\theta_X(\tau/C)\]
for all $k\in \{k_0+1,\ldots,k_0+N\}$ and all $z^*\in Z$ with $\|z^*\|\geq \tau/L$. Hence, adding those inequalities, we obtain that 
\[ |y^*+z^*| \geq  |y^* |-\frac{2}{N}|y^*|_{k_0+1}+\bar\theta_X(\tau/C).\]
 Since $\frac{2}{N}|y^*|_{k_0+1}\leq \frac{2L}{N}\|y^*\|\leq \frac{2L^2}{N}<\delta$, we are done. 
\end{proof}

We point out a quantitative strengthening of Theorem \ref{ThmTheta*Pres} which we will need. Precisely, we notice that our proof of Theorem \ref{ThmTheta*Pres} allows us to change the norm of $X$ by a fixed amount without having to change the constants $L$ and $C$.

\begin{theorem}
\label{ThmTheta*PresEquivNorm}
Let $X$ be an AUS Banach space and let  $Y$ be a Banach space   asymptotically coarse Lipschitz equivalent to $X$. For all $K\geq 1$, there are $L=L(K),\ C=C(K)\geq 1$ satisfying the following: if $|\cdot|_X$ is a norm on $X$ with $|\cdot|_X\sim_K\|\cdot\|_X$ and  $\delta \in (0,1)$, then there is a norm $|\cdot|_Y$ on $Y$   such that $|\cdot|\sim_L\|\cdot\|_Y$ and 
\[\bar\theta_{(Y,|\cdot|_Y)}(\tau)\geq \bar\theta_{(X,|\cdot|_X)}(\tau/C)-\delta\ \text{ for all }\ \tau\in (0,1).\]
\end{theorem}

\begin{proof}
 Let $f:X\to Y$ be an asymptotic coarse Lipschitz equivalence with asymptotic coarse Lipschtiz inverse $g:Y\to X$. Examining the proof of Theorem \ref{ThmTheta*Pres}, we see that  $L$  only depends on the affine upper bounds for the asymptotic coarse Lipschitz equivalences $f$ and $g$. Therefore, if we renorm $X$ with a $K$-equivalent norm, we can replace $L$ with $LK$ in the proof of Theorem \ref{ThmTheta*Pres} and the result will still hold.

As for $C$, this constant is taken in the proof of Theorem \ref{ThmTheta*Pres} to be any which satisfies \eqref{EqC}. Notice that $s$ depends on $L$ only, so $C$ depends only on $L$ and $M$, where $M$ is given by Theorem \ref{ThmGorelik}. The result then  follows from    Remark \ref{RemarkConstantGorelik}.
\end{proof}

We now translate this result in terms of the modulus of asymptotic smoothness.

\begin{theorem}
\label{ThmRhoPresEquivNorm}
Let $X$ be an AUS Banach space and let  $Y$ be a Banach space   asymptotically coarse Lipschitz equivalent to $X$. For all $K\geq 1$, there are $L=L(K),\ C=C(K)\geq 1$ satisfying the following: if $|\cdot|_X$ is a norm on $X$ with $|\cdot|_X\sim_K\|\cdot\|_X$ and  $\delta \in (0,1)$, then there is a norm $|\cdot|_Y$ on $Y$   such that $|\cdot|\sim_L\|\cdot\|_Y$ and 
\[\bar\rho_{(Y,|\cdot|_Y)}(\tau/C)\leq \bar\rho_{(X,|\cdot|_X)}(\tau)+\delta\ \text{ for all }\ \tau\in (0,1).\]
\end{theorem}

\begin{proof} Let $\varphi,\psi$ be continuous monotone non decreasing on $[0,1]$ with $\varphi(0)=\psi(0)=0$. If there exists $D\ge 1$ and $\delta >0$ such that for all $\tau\in (0,1)$, $\varphi(\tau)\ge \psi(\tau/D)-\delta$, it is easy to  show that for all $\tau\in (0,1)$, $\varphi^*(\tau/D)\le \psi^*(\tau)+\delta$. So the conclusion follows clearly from Theorem \ref{ThmTheta*PresEquivNorm} and Proposition    \ref{PropositionPhoThetaEquiv}.
\end{proof}

\subsection{Preservation of asymptotic structures}

We are almost ready to prove our results on the preservation of asymptotic structures. For this we will exploit the following two renorming characterizations of the classes. 

\begin{theorem}\emph{(}\cite[Corollary 4.15]{Causey2018ThreeAndHalf}\emph{).} Let $X$ be a Banach space and $p\in (1,\infty]$. Then $X$ has property $\textsf{\emph{N}}_p$ if and only if there exist a constant $L\ge 1$ and a constant $C>0$, such that for any $\tau \in (0,1]$ there exists a norm $|\cdot |$ on $X$ such that $|\cdot|\sim_L \|\cdot\|_X$  and
\begin{enumerate}[(a)]
\item if $1<p<\infty$, $\overline{\rho}_{|\cdot |}(\tau)\leqslant C \tau^p$ 
\item if $p=\infty$, $\overline{\rho}_{|\cdot |}(C) \leqslant \tau$. 
\end{enumerate}  \label{TheoremCauseyNp}
\end{theorem}

\begin{theorem}\emph{(}\cite[Theorem A]{CauseyFovelleLancien2021}\emph{).} Let $X$ be a Banach space and $p\in (1,\infty)$. Then $X$ has property $\textsf{\emph{A}}_p$ if and only if there exist a constant $L\ge 1$ and a constant $C>0$, such that for any $\tau \in (0,1]$ there exists a norm $|\cdot |$ on $X$ such that $|\cdot|\sim_L \|\cdot\|_X$  and
$$\forall t\ge \tau,\ \ \overline{\rho}_{|\cdot |}(t)\le Ct^p.$$\label{TheoremCFLAp}
\end{theorem}

We can now give the proof of Theorem \ref{precise stability}, which will follow from Theorem \ref{ThmRhoPresEquivNorm} and the above renorming characterizations of $\textsf{{A}}_p$ and $\textsf{{N}}_p$, for $p\in (1,\infty)$ and of $\textsf{{A}}_\infty=\textsf{{N}}_\infty$.

\begin{proof}[Proof of Theorem \ref{precise stability}] We only detail the case of $\textsf{{A}}_p$ for $p\in (1,\infty)$. So assume that $X$ has $\textsf{{A}}_p$ and that $Y$ is asymptotically coarse Lipschitz equivalent to $X$. As $X$ has property $\textsf{{A}}_p$, let $L '\geq 1$ and  $C'>0$ be given by Theorem \ref{TheoremCFLAp}, i.e.,   for all $\tau\in (0,1]$, we can find an $L'$-equivalent   renorming $|\cdot|_X$ of  $X$ so that  \begin{equation}\label{Eq.Ap.2}
\bar \rho_{(X,|\cdot|_X)}(t)\leq C't^p\ \text{ for all }\ t\geq \tau.
\end{equation}
Let $L=L(L'),C=C(L')\geq 1$ be given by Theorem \ref{ThmRhoPresEquivNorm}, i.e., if $|\cdot|_X$ is an $L'$-equivalent renorming of $X$ and  $\delta>0$, then  there is a norm  $|\cdot|_\delta$ on $Y$ such that  $|\cdot|_\delta\sim_L \|\cdot\|_Y$
and 
\begin{equation}\label{Eq.Ap.1.Young.Func}
\bar\rho_{(Y,|\cdot|_\delta)}(t/C)\leq \bar\rho_{(X,|\cdot|_X)}(t) +\delta\ \text{ for all }\ t\in (0,1).
\end{equation}
We now show that $Y$ satisfies the condition in Theorem \ref{TheoremCFLAp} with constants $L$ and $(C'+1)(12C)^p$.  For that, fix  $\tau\in (0,1]$. Let $|\cdot|_X$ be an $L'$-equivalent  norm on $X$ satisfying \eqref{Eq.Ap.2} and $|\cdot|_\delta$ a norm on $Y$ satisfying \eqref{Eq.Ap.1.Young.Func}. Then, 

\begin{align*}
\bar \rho_{(Y,|\cdot|_\delta)}\Big(\frac{t}{C}\Big) \le C't^p+\delta  
\end{align*}
for all $\delta>0$ and all $t\in [\tau,1]$. Therefore, choosing $\delta\in (0,\tau^p)$, we get that 
\[\bar \rho_{(Y,|\cdot|_\delta)}(t)\leq 2C'C^pt^p,\]
for all $t\ge \tau C$. As $\tau>0$ was arbitrary, we obtain that $Y$ satisfies the condition in Theorem \ref{TheoremCFLAp}, i.e., $Y$ has property $\textsf{{A}}_p$.

The case of $\textsf{{N}}_p$  follows analogously   with Theorem \ref{TheoremCauseyNp} replacing Theorem \ref{TheoremCFLAp} above.
\end{proof}

\begin{proof}[Proof of Theorem \ref{ThmAUSPreservationINTRO}]
    This follows immediately from   Theorems \ref{inclusions} and \ref{precise stability}. Indeed, suppose $p\in (1,\infty)$ and let $X$ be a $p$-AUS Banach space. Then, by Theorem \ref{inclusions} $X$ is in $\textsf{{A}}_p$. Then, if $Y$ is asymptotically coarse Lipschitz equivalent to $X$, Theorem \ref{precise stability} implies that $Y$ is also in $\textsf{{A}}_p$. Applying Theorem \ref{inclusions} once again, we conclude that $Y$ is $p'$-AUSable for all $p'\in (1,p)$.
\end{proof}

\end{document}